\documentclass[12pt,a4paper]{article}
\usepackage[centertags]{amsmath}
\usepackage{amsfonts}
\usepackage{amssymb}
\usepackage{amsthm}
\usepackage{color}
\usepackage{marginnote}
\usepackage{float}

\usepackage{fullpage}
\usepackage{tikz}
\usepackage{verbatim}
\usetikzlibrary{arrows}

\newtheorem{cor}{Corollary}
\newtheorem{defi}{Definition}
\newtheorem{ex}{Example}
\newtheorem{rem}{Remark}
\newtheorem{prop}{Proposition}
\newtheorem {lem}{Lemma}

\newtheorem{teo}{Theorem}

\usepackage{color}

\begin{document}

\title{Modal operators for meet-complemented lattices}

\author{Jos\'{e} Luis Castiglioni and Rodolfo C. Ertola-Biraben}


\maketitle

\begin{abstract}
We investigate some modal operators of necessity and possibility in the context of
meet-complemented (not necessarily distributive) lattices. We proceed in stages.
We compare our operators with others.
\end{abstract}

\section*{Introduction}

In this paper we consider certain modal operators of necessity and possibility
in the context of a (not necessarily distributive) meet-complemented lattice.
Our operators are sort of relatives of modal operators to our best knowledge first studied by Moisil in 1942
(see \cite{Moi1} or \cite{Moi2}).
He worked in a logical context were he had both intuitionistic negation $\neg$ and its dual, which we annotate $D$,
and he used $DD$ for necessity and $\neg \neg$ for possibility.
This choice we have found somewhat intriguing as, for instance, the normal modal inequality is not the case,
that is, we do not have $DD(x \to y) \leq DDx \to DDy$.
Also, because he could have chosen $\neg D$ and $D \neg$ for necessity and possibility, respectively,
operations that satisfy the just given normal modal inequality.
In 1974, Rauszer (see \cite{R}) considered lattices expanded with both the meet and the join relative complements.
In those algebras both $\neg$ and $D$ are easily definable.
However, she does not mention Moisil.
Also, she does not seem to be interested in necessity or possibility.
Later on, in 1985, L\'{o}pez-Escobar, who seems not to have been acquainted with Moisil's paper,
considered, in the context of Beth structures,
modal operators of necessity and possibility, $\neg D$ and $D \neg$, respectively.
For more or less recent papers on intuitionistic modal logic, see \cite{S} and \cite{BdP}.

In this paper we define modal operators of necessity and possibility that are very similar to
the mentioned $\neg D$ and $D \neg$.
Our operators are defined as maximum and minimum, respectively, so, they are univocal, i.e., when they exist,
there cannot be two different operations satisfying their definitions.
Another aspect of our operators is that their definition does not require $D$.
Indeed, in order to define them, it is enough to have a (not necessarily distributive) meet-complemented lattice, that is,
the usual algebraic counterpart of the connectives of conjunction, disjunction, and negation in intuitionistic logic.
The relative meet-complement, i.e. the algebraic counterpart of intuitionistic conditional, is also not needed.

The paper is organized as follows. In Section 1 we recall certain facts regarding meet-complemented lattices.
Since in this context, as already noted by Frink (see \cite{Fri}), distributivity is not forced,
we will only assume it in sections 5 and 7.
In Section 2 we expand meet-complemented lattices with necessity $\Box$ and
prove that the expansion is an equational class.
In Section 3 we expand meet-complemented lattices with possibility $\Diamond$ and
prove that the expansion is not an equational class.
In Section 4 we expand meet-complemented lattices with both $\Box$ and $\Diamond$ and
prove that the expansion is, again, an equational class.
In Section 5 we consider the distributive extension.
In Section 6 we consider the extension given by the modal logic S4-Schema, that is, $\Box \Box = \Box$.
In Section 7 we extend with both distributivity and the S4-Schema.
Finally, in Section 8 we add the relative meet-complement.
We compare $\Box$ and $\Diamond$ with similar operators.

We talk of \emph{extensions} of a class of algebras when we only add some (new) property to the operations in the given class,
for instance, distributivity.
On the other hand, we talk of \emph{expansions} when adding a (new) operation to the class,
for example, when we add necessity $\Box$ to meet-complemented lattices.

In this paper we do not consider the two logics involved (see e.g \cite{Font}).

\section{Lattices with meet-complement}

As usual, a lattice ${\bf{L}} = (L, \leq)$ will be a non-empty ordered set such that for all $a,b \in L$ there exist 
$a \wedge b$ and $a \vee b$ that satisfy the following facts:

\vspace{5pt}

{\bf{($\wedge$I)}} if $c \leq a$ and $c \leq b$, then $c \leq a \wedge b$, 

{\bf{($\wedge$E)}} $a \wedge b \leq a, b$, 

{\bf{($\vee$I)}} $a, b \leq a \vee b$, 

{\bf{($\vee$E)}} if $a \leq c$ and $b \leq c$, then $a \vee b \leq c$. 

\vspace{5pt}

\noindent Let us also remind that, even not having distributivity, we still have the following facts.

\begin{lem}
 Let $\bf{L}$ be a lattice. Then, for all $a, b, c \in L$, we have

 \noindent (i)  Distributivity of $\vee$ relative to $\wedge$: $a \vee (b \wedge c) \leq (a \vee b) \wedge (a \vee c)$,

 \noindent (ii) Factoring of infima: $(a \wedge b) \vee (a \wedge c) \leq a \wedge (b \vee c)$,

 \noindent (iii) Monotonicity of $\vee$: if $a \leq b$, then $c \vee a \leq c \vee b$.
\end{lem}


\begin{cor} \label{cL1}
 Let $\bf{L}$ be a lattice with top 1.
 Then, for all $a, b, c \in L$, if $a \vee b = 1$ and $b \leq c$, then $a \vee c = 1$.
\end{cor}


Let us expand a lattice $\bf{L}$ postulating the existence of the meet-complement
$\neg a = max \{b \in L: a \wedge b \leq c$, for all $c \in L\}$.
Ribenboim and Balbes (see \cite{Rib} and \cite{BalDwi}, respectively) suppose that {\bf{L}} has a bottom element.
However, it is immediate that we have

\smallskip

{\bf{($\neg$E)}} $a \wedge \neg a \leq b$, for all $a, b \in L$, for every lattice $\bf{L}$ with meet-complement.

\smallskip

\noindent We also have

\smallskip

{\bf{($\neg$I)}} for any $a, b \in L$, if $a \wedge b \leq c$, for all $c \in L$, then $b \leq \neg a$.

\smallskip

We will use the notation $\mathbb{ML}$ for the class of lattices with meet-complement.
Throughout the paper we will use the following examples of meet-complemented lattices.

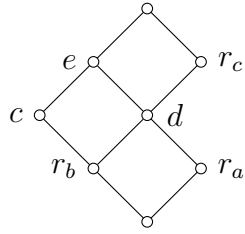
\begin{figure} [ht]
\begin{center}
\begin{tikzpicture}

    \tikzstyle{every node}=[draw, circle, fill=white, minimum size=4pt, inner sep=0pt, label distance=1mm]

    \draw (0,0) node (1) [] {}
        -- ++(315:1cm) node (rc) [label=right:$r_c$] {}
        -- ++(225:1cm) node (d) [label=right:$d$] {}
        -- ++(315:1cm) node (ra) [label=right:$r_a$] {}
        -- ++(225:1cm) node (0) [label=left:] {}
        -- ++(135:1cm) node (rb) [label=left:$r_b$] {}
        -- ++(135:1cm) node (c) [label=left:$c$] {}
        -- ++(45:1cm) node (e) [label=left:$e$] {}
        -- (1);


        \draw (d) -- (e);
        \draw (d) -- (rb);

\end{tikzpicture}
\end{center}
\caption{\label{rn8} The lattice {\bf{$R_8$}}}
\end{figure}

\begin{ex}
 (i) The three-element chain will be annotated $3$ and its middle element $m$.

 \noindent (ii) The five-element $2^{2} \oplus 1$. Any of its atoms will be annotated $l$.

 \noindent (iii) The 8-element lattice $R_8$,
 which coincides with the lattice of the first 8-elements of the Rieger-Nishimura lattice (see Figure \ref{rn8}).
 We will denote $r_b$, $r_a$, and $r_c$ the meet-reducible atom, meet-irreducible atom, and join-irreducible coatom of $R_8$,
 respectively.

 \noindent (iv) The 13-element lattice that appears in Figure \ref{13}.

 \noindent (v) The five-element modular and non-distributive lattice sometimes denoted $M_5$,
 which we shall call ``the pentagon'' (see Figure \ref{M_5}).
 We will also denote $a, b, c$ the atom non-coatom, the atom-coatom, and the coatom non-atom of the pentagon, respectively.
\end{ex}

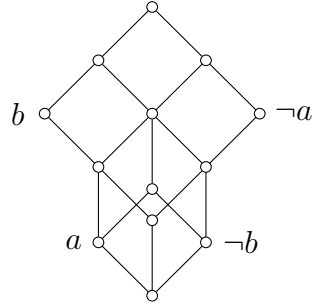
\begin{figure} [ht]
\begin{center}

\begin{tikzpicture}

    \tikzstyle{every node}=[draw, circle, fill=white, minimum size=4pt, inner sep=0pt, label distance=1mm]

    \draw (0,0) node (1) [] {}
        -- ++(225:1cm) node (2) [] {}
        -- ++(225:1cm) node (b) [label=left:$b$] {}
        -- ++(315:1cm) node (4) [] {}
        -- ++(45:1cm) node (5) [] {}
        -- ++(45:1cm) node (6) [] {}
        -- ++(315:1cm) node (na) [label=right:$\neg a$] {}
        -- ++(225:1cm) node (8) [] {}

        -- ++(270:1cm) node (nb) [label=right:$\neg b$] {}
        -- ++(135:1cm) node (10) [] {}
        -- ++(225:1cm) node (a) [label=left:$a$] {}
        -- ++(315:1cm) node (12) [] {}
        -- ++(90:1cm) node (13) [] {}

        -- (4);


        \draw (1) -- (6);
        \draw (2) -- (5);
        \draw (8) -- (5);

        \draw (12) -- (nb);
        \draw (4) -- (a);
        \draw (8) -- (13);
        \draw (5) -- (10);

\end{tikzpicture}

\end{center}
\caption{\label{13} The lattice {\bf{$13$}}}
\end{figure}

\begin{figure} [ht]
\begin{center}

\begin{tikzpicture}

    \tikzstyle{every node}=[draw, circle, fill=white, minimum size=3pt, inner sep=0pt, label distance=1mm]

    \draw (0,0) node (1) [] {}
        -- ++(225:1cm) node (c) [label=left:$c$] {}
        -- ++(270:1cm) node (a) [label=left:$a$] {}
        -- ++(315:1cm) node (0) [] {}
        -- ++(60:1.4142cm) node (b) [label=right:$b$] {}
        -- (1);

\end{tikzpicture}

\end{center}
\caption{\label{M_5} The pentagon}
\end{figure}
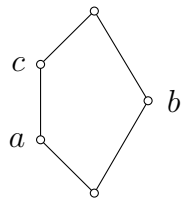

\noindent Apart from having a bottom element that will be denoted with $0$, 
any meet-complemented lattice has a top element $1$ taking $\neg(a \wedge \neg a)$, for any $a \in L$.
So, every meet-complemented lattice is bounded and we have both $\neg 0 = 1$ and $\neg 1 = 0$.

Regarding properties only involving $\neg$, we get {\bf{(DN)}}: $a \leq \neg \neg a$.
We also have that $\neg$ is antimonotonic: if $a \leq b$, then $\neg b \leq \neg a$.
As a corollary of both facts, we get {\bf{(TN)}}: $\neg \neg \neg a = \neg a$ and if $a \leq \neg b$, then $b \leq \neg a$.
If we also consider $\wedge$, we have $\neg \neg (a \wedge b)= \neg \neg a \wedge \neg \neg b$.


Note that $\neg$ does not exist in the modular and non-distributive five-element lattice sometimes denoted as $M_5$,
which we shall call ``the diamond''.
However, $\neg$ does exist in the pentagon.
So, the existence of $\neg$ does not imply modularity.
Then, it neither implies distributivity, as distributivity implies modularity (see \cite[p. 11]{BurSan}).
In this respect, $\neg$ differs from the relative meet-complement, which implies distributivity,
as was already noted by Skolem in 1919 (see \cite {Sko1} or \cite{Sko2}).

Regarding equations or inequalities involving $\vee$,
as corollaries of $\neg$-antimonotonicity, we get the following De Morgan inequalities:
$\neg (a \vee b) \leq \neg a \wedge \neg b$ and $\neg a \vee \neg b \leq \neg (a \wedge b)$.
The other De Morgan inequality valid in Heyting algebras, i.e. $\neg a \wedge \neg b \leq \neg (a \vee b)$, also holds.
Proof: We have that both $a$ and $b$ are less or equal to $\neg(\neg a \wedge \neg b)$.
Then, $a \vee b \leq \neg(\neg a \wedge \neg b)$.
By antimonotonicity of $\neg$, it follows that $\neg \neg (\neg a \wedge \neg b) \leq \neg (a \vee b)$.
So, $\neg a \wedge \neg b \leq \neg (a \vee b)$.
Consequently, distributivity is not needed in order to get any of the De Morgan inequalities valid for Heyting algebras.
For more details concerning not necessarily distributive meet-complemented lattices, see \cite {Fri}.

The inequality $(a \vee \neg a) \wedge \neg \neg a \leq a$ is not valid in the pentagon, taking $a$ to be the non-coatom atom.
So, the inequality $(a \vee b) \wedge \neg b \leq a$ is also not the case,
as can also be seen in the pentagon taking $a$ and $b$ to be the non-coatom atom and the coatom atom, respectively.
The inequality $(a \vee \neg b) \wedge b \leq a$ is also not the case,
taking in the pentagon $a$ and $b$ to be the non-coatom atom and the non-atom-coatom, respectively.
However, we do have the following facts.

\begin{lem} \label{lmc}
 Let {\bf{L}}$=(L; \wedge, \vee, \neg)$ be a meet-complemented lattice.
 Then, for all $a, b \in L$,

 (i) $(a \vee b) \wedge \neg b \leq \neg \neg a$,
 (ii) $(a \vee \neg b) \wedge b \leq \neg \neg a$,
 (iii) $(a \vee \neg b) \wedge \neg a \leq \neg b$,

 (iv) if $a \vee b = 1$, then $\neg b \leq \neg \neg a$,

 (v) if $a \vee \neg b = 1$, then $b \leq \neg \neg a$,

 (vi) if $a \vee \neg b=1$, then $\neg a \leq \neg b$,

 (vii) $\neg a=1$ iff $a=0$, and

 (viii) $a \vee \neg a=a$ iff $\neg a=0$.
\end{lem}

\begin{proof}
 \noindent (i) Using ($\neg$I), it is enough to get that $\neg a \wedge ((a \vee b) \wedge \neg b) \leq c$,
 which follows using ($\neg$E), because we have seen we have $\neg a \wedge \neg b \leq \neg (a \vee b)$.
 (ii) and (iii) follow similarly, noting that we also have $\neg a \wedge b \leq \neg (a \vee \neg b)$.

 \noindent (iv), (v), and (vi) follow from (i), (ii), and (iii), respectively.

 \noindent (vii) and (viii) are easy to check.
\end{proof}

\begin{prop}
 There are finite lattices where $\neg$ does not exist.
\end{prop}

\begin{proof}
 Just consider any atom (or coatom) in the diamond.
\end{proof}


It is natural to ask whether $\mathbb{ML}$ is an equational class.
Ribenboim (see \cite{Rib}) answered positively in the distributive case.
Balbes and Dwinger (see \cite {BalDwi}) also state the solution in the distributive case.
In fact, their solution also holds in the non-distributive case.
Using the abbreviation $x \preccurlyeq y$ for $x \wedge y \approx x$,
it is enough to take the identities proving that the class of lattices is an equational class and add the following:

\smallskip

{\bf{($\neg$E)}} $x \wedge \neg x \preccurlyeq y$,

\smallskip

{\bf{($\neg$I1)}} $y \preccurlyeq \neg (x \wedge \neg x)$,

\smallskip

{\bf{($\neg$I2)}} $x \wedge \neg(x \wedge y) \preccurlyeq \neg y$.

\smallskip

\noindent In passing, note that $\vee$ is not needed in order to prove that $\mathbb{ML}$ is an equational class.

We will occasionally deal with Boolean elements.

\begin{defi}
An element $a$ of the universe of a lattice ${\bf{L}}$ is called \emph{complemented} iff
there is an element $b \in A$ such that $a \wedge b=0$ and $a \vee b = 1$.

An element $a$ in the universe of a meet-complemented lattice is called \emph{Boolean} iff $a \vee \neg a =1$
\end{defi}

Let us see that the concepts of complemented and of Boolean element agree in meet-complemented lattices.

\begin{lem} \label{BL}
An element $a$ in the universe of a meet-complemented lattice is complemented iff $a \vee \neg a = 1$.
\end{lem}

\begin{proof}
 $\Rightarrow)$ Suppose there is a $b$ such that (i) $a \wedge b = 0$ and (ii) $a \vee b = 1$.
 First, using (i), we have that $b \leq \neg a$, which using (ii) and Corollary \ref{cL1} gives $a \vee \neg a = 1$.

\noindent $\Leftarrow)$ Suppose $a \vee \neg a = 1$.
 The goal follows because $a \wedge \neg a = 0$.
\end{proof}


\section{Meet-complemented lattices with necessity}

Let us consider a meet-complemented lattice ${\bf{A}}$ where there exists necessity defined as
$\Box a=max\{b \in A: a \vee \neg b=1 \}$, for any $a \in A$.
It is equivalent to state both

\smallskip

{\bf{($\Box$E)}} $a \vee \neg \Box a=1$ and

\smallskip

{\bf{($\Box$I)}} if $a \vee \neg b=1$, then $b \leq \Box a$.

\begin{rem}
 Note that $\Box \emph{E}$ is equivalent to the reciprocal of $\Box \emph{I}$.
 So, $\Box \emph{E}$ and $\Box \emph{I}$ taken together are equivalent to saying that
 $b \leq \Box a$ iff $a \vee \neg b=1$.
 Also, note that $\Box \emph{I}$ is equivalent to saying that if $a \vee b=1$, then $\neg b \leq \Box a$.
\end{rem}

Let us use the notation ${\mathbb{ML^\Box}}$ for the class of meet-complemented lattices with $\Box$.

\begin{lem}
Let $\bf{A} \in \mathbb{ML^\Box}$. Then, we have

 \noindent $\Box$-Monotonicity: if $a \leq b$, then $\Box a \leq \Box b$, for any $a, b \in A$.
\end{lem}

\begin{proof}
By ($\Box$E) we have $a \vee \neg \Box a=1$.
Using $a \leq b$, it follows that $b \vee \neg \Box a=1$, which, using ($\Box$I), gives $\Box a \leq \Box b$.
\end{proof}

We will be particularly interested in modalities, i.e. finite combinations of unary operators,
at the present stage $\neg$ and $\Box$.
As modalities only involve one argument, we will omit it as much as possible.
Firstly, note that we have the following fact.

\begin{lem} \label{(DN)B}
Let us consider a lattice of $\mathbb{ML^\Box}$. Then,
 $\neg \neg \Box = \Box$.
\end{lem}

\begin{proof}
 Using ($\Box$E) we have $a \vee \neg \Box a = 1$.
 It follows by (TN) that $a \vee \neg \neg \neg \Box a = 1$,
 which using ($\Box$I) gives $\neg \neg \Box \leq \Box$.
 The other inequality follows using (DN).
\end{proof}

We will use $\circ$ for the identical modality.
Also, we will distinguish between \emph{positive} and \emph{negative} modalities.
However, we will mix results on both sorts of modalities in order to simplify proofs,
for instance in the following case.

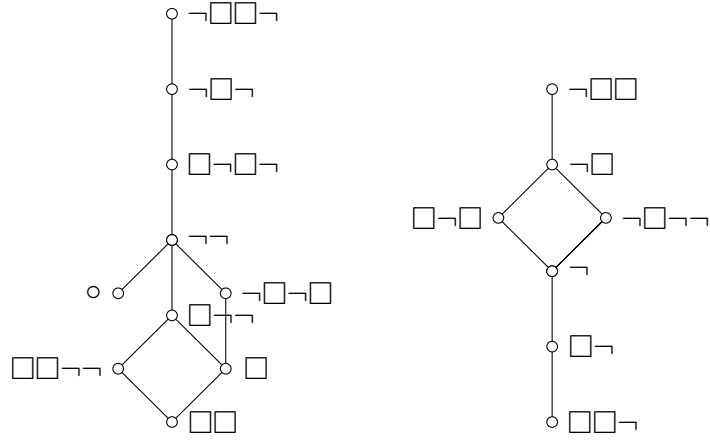
\begin{figure} [ht]
\begin{center}

\begin{tikzpicture}

    \tikzstyle{every node}=[draw, circle, fill=white, minimum size=4pt, inner sep=0pt, label distance=1mm]

    \draw (0,0) node (1) [label=right:$\neg \Box \Box \neg$] {}
        -- ++(270:1cm) node (nBnB) [label=right:$\neg \Box \neg$] {}
        -- ++(270:1cm) node (nBn) [label=right:$\Box \neg \Box \neg$] {}
        -- ++(270:1cm) node (nn) [label=right:$\neg \neg$] {}
        -- ++(315:1cm) node (nBnB) [label=right:$\neg \Box \neg \Box$] {}
        -- ++(270:1cm) node (B) [label=right:$\Box$] {}
        -- ++(225:1cm) node (BB) [label=right:$\Box \Box$] {}
        -- ++(135:1cm) node (BBnn) [label=left:$\Box \Box \neg \neg$] {}
        -- ++(45:1cm) node (Bnn) [label=right:$\Box \neg \neg$] {}
        -- ++(90:1cm) node (nn) [] {}
        -- ++(225:1cm) node (c) [label=left:$\circ$] {};

        \draw (Bnn) -- (B);

        \draw (5,-1) node (nBB) [label=right:$\neg \Box \Box$] {}
        -- ++(270:1cm) node (nB) [label=right:$\neg \Box$] {}
        -- ++(225:1cm) node (BnB) [label=left:$\Box \neg \Box$] {}
        -- ++(315:1cm) node (n) [label=right:$\neg$] {}
        -- ++(45:1cm) node (nBnn) [label=right:$\neg \Box \neg \neg$] {}
        -- ++(225:1cm) node (n) [] {}
        -- ++(270:1cm) node (Bn) [label=right:$\Box \neg$] {}
        -- ++(270:1cm) node (BBn) [label=right:$\Box \Box \neg$] {};

        \draw (nBnn) -- (nB);

\end{tikzpicture}

\end{center}
\caption{\label{pmb} Positive and negative modalities for $\neg$ and $\Box$ for up to two boxes}
\end{figure}




%



\begin{lem} \label{lmcln}
Let us consider a lattice of $\mathbb{ML^\Box}$. Then,
 (i) $\circ \leq \neg \neg$,
 (ii) $\Box \leq \Box \neg \neg$,
 (iii) $\neg \Box \neg \neg \leq \neg \Box$,
 (iv) $\Box \neg \leq \neg$,
 (v) $\Box \neg \Box \leq \neg \Box$,
 (vi) $\Box \neg \Box \neg \leq \neg \Box \neg$,
 (vii) $\Box \neg \neg \leq \neg \neg$,
 (viii) $\neg \leq \neg \Box \neg \neg$,
 (ix) $\neg \leq \Box \neg \Box$,
 (x) $\neg \neg \leq \Box \neg \Box \neg$,
 (xi) $\Box \Box \leq \Box$,
 (xii) $\Box \Box \neg \leq \Box \neg$,
 (xiii) $\neg \Box \leq \neg \Box \Box$,
 (xiv) $\neg \Box \neg \leq \neg \Box \Box \neg$.
\end{lem}

\begin{proof}
 (i) is (DN).

 \noindent (ii) follows from (i) using $\Box$-monotonicity.

 \noindent (iii) follows from (ii) by $\neg$-antimonotonicity.

 \noindent (iv) Using Lemma \ref{lmc}(iii) we have
 $(\neg a \vee \neg \Box \neg a) \wedge \neg \neg \Box \neg a \leq \neg a$, for any $a$.
 Now, by ($\Box$E) we have $\neg a \vee \neg \Box \neg a = 1$.
 So, $\neg \neg \Box \neg a \leq \neg a$.
 Using (DN) it follows that $\Box \neg a \leq \neg a$.

 \noindent (v), (vi), and (vii) follow from (iv).

 \noindent (viii) follows from (vii) by $\neg$-antimonotonicity and (TN).

 \noindent (ix) follows from ($\Box$E).

 \noindent (x) follows from ($\Box$E) and (TN).

 \noindent (xi) Using Lemma \ref{lmc}(ii), we have
 $(\Box a \vee \neg \Box \Box a) \wedge \Box \Box a \leq \neg \neg \Box a$, for any $a$.
 Now, by ($\Box$E) we have $(\Box a \vee \neg \Box \Box a) =1$.
 So, we get $\Box \Box \leq \neg \neg \Box$.
 It follows by Lemma \ref{(DN)B} that $\Box \Box \leq \Box$.

 \noindent (xii) follows from (xi).

 \noindent (xiii) follows from (xi) by $\neg$-antimonotonicity.

 \noindent (xiv) follows from (xii) by $\neg$-antimonotonicity.
\end{proof}

It can also be seen that the reverse inequalities are not the case.
For that purpose, as already said,
$m$ will be the atom-coatom of the three-element chain, $l$ any of the atoms of the lattice $2^{2} \oplus 1$, and
$r_c$ the join-irreducible coatom of $R_8$.
We will also denote $a, b, c$ the atom non-coatom, the atom-coatom, and the coatom non-atom of the pentagon, respectively.

\begin{figure} [ht]
\begin{center}

\begin{tikzpicture}

    \tikzstyle{every node}=[draw, circle, fill=white, minimum size=4pt, inner sep=0pt, label distance=1mm]

    \draw (0,0) node (1) [label=right:$\neg \Box \Box \neg a$] {}
        -- ++(225:1cm) node (2) [] {}
        -- ++(225:1cm) node (b) [label=left:$\neg \Box \neg a$] {}
        -- ++(315:1cm) node (4) [] {}
        -- ++(45:1cm) node (5) [] {}
        -- ++(45:1cm) node (6) [] {}
        -- ++(315:1cm) node (na) [label=right:$\neg a$] {}
        -- ++(225:1cm) node (8) [] {}

        -- ++(270:1cm) node (nb) [label=right:$\Box \neg a$] {}
        -- ++(135:1cm) node (10) [] {}
        -- ++(225:1cm) node (a) [label=left:$a$] {}
        -- ++(315:1cm) node (12) [label=right:$\Box \Box \neg a$] {}
        -- ++(90:1cm) node (13) [] {}

        -- (4);


        \draw (1) -- (6);
        \draw (2) -- (5);
        \draw (8) -- (5);
        \draw (12) -- (nb);
        \draw (4) -- (a);
        \draw (8) -- (13);
        \draw (5) -- (10);

\end{tikzpicture}

\end{center}
\caption{\label{b13} Behaviour of $\Box$ and $\neg$ in an atom of $13$}
\end{figure}
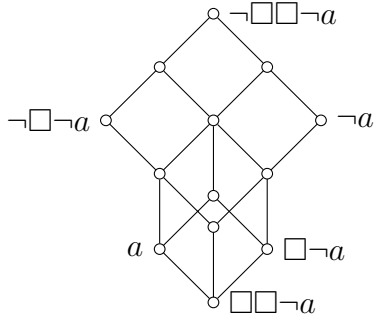

We now give the nodes of the algebras proving that none of the inverse inequalities of Lemma \ref{lmcln} holds:
For (i)-(iii) and (ix): $m$.
For (iv), (vii), (viii), and (x): $l$.
For (v), (xi), and (xiii): $r_c$.
For (vi), (xii), and (xiv): check behaviour of atom $a$ in Figure \ref{b13}, where $\Box \neg \Box \neg a = a$.

We also need to show the cases where the elements are non-comparable.
That is, using the symbol $\parallel$ for non-comparability,
we need to show both $\Box \neg \Box \parallel \neg \Box \neg \neg$ and $\circ \parallel \Box \Box$, $\Box$, $\Box \neg \neg$.
To see that neither $\Box \neg \Box \nleq \neg \Box \neg \neg$ nor $\neg \Box \neg \neg \nleq \Box \neg \Box$ consider $m$ and
the node labeled $\neg \Box \neg a$ in Figure \ref{b13}, respectively.
To prove the other non-comparable cases, it is enough, using $\leq$-transitivity,
to see that $\circ \nleq \Box \neg \neg$ and $\Box \Box \nleq \circ$, which hold using $l$ and $a$, respectively.

Regarding the interaction of $\Box$ with $\wedge$, $\vee$, $\neg$, $0$, and $1$, consider the following facts.

\begin{prop} \label{BoxIn}
  Let ${\bf{A}} \in \mathbb{ML^\Box}$ and let $a, b \in A$. Then,
 (i) $\Box(a \wedge b) \leq \Box a \wedge \Box b$,
 (ii) $\Box a \vee \Box b \leq \Box (a \vee b)$,
 (iii) $\Box (a \vee \neg b) \wedge b \leq \Box a$,
 (iv) $\Box (a \vee b) \wedge \neg b \leq \Box a$,
 (v) $\Box a \wedge \Box \neg a = 0$,
 (vi) $\Box 0 = 0$, and
 (vii) $\Box a = 1$ iff $a = 1$.
\end{prop}

\begin{proof}
 (i) follows by $\Box$-monotonicity using that $a \wedge b \leq a$.

 \noindent (ii) follows by $\Box$-monotonicity using that $a \leq a \vee b$.

 \noindent (iii) By ($\Box$E) we have $(a \vee \neg b) \vee \neg \Box (a \vee \neg b) = 1$.
 By $\vee$-associativity, we get $a \vee (\neg b \vee \neg \Box(a \vee \neg b)) = 1$.
 By De Morgan and commutativity, $a \vee \neg (\Box (a \vee \neg b) \wedge b) = 1$ and,
 by ($\Box$I), we finally get $\Box(a \vee \neg b) \wedge b \leq \Box a$.

 \noindent (iv) By ($\Box$E) we have $(a \vee b) \vee \neg \Box(a \vee b) = 1$.
 By (DN), $(a \vee \neg \neg b) \vee \neg \Box (a \vee b) = 1$.
 By $\vee$-associativity, we get $a \vee (\neg \neg b \vee \neg \Box (a \vee b)) = 1$.
 By De Morgan, $a \vee \neg(\neg b \wedge \Box (a \vee b)) = 1$. 
 By ($\Box$I) we get $\neg b \wedge \Box (a \vee b) \leq \Box a$. 
 Finally, by $\wedge$-commutativity, $\Box (a \vee b) \wedge \neg b \leq \Box a$. 

 \noindent (v) follows from Lemma \ref{lmcln}(ii), (iv), and (vi) together with the fact that $\neg \neg a \wedge \neg a = 0$.

 \noindent (vi) is easy to check.

 \noindent (vii) follows easily using Lemma 1(vii).
\end{proof}

\begin{figure} [ht]
\begin{center}

\begin{tikzpicture}

    \tikzstyle{every node}=[draw, circle, fill=white, minimum size=4pt, inner sep=0pt, label distance=1mm]


    \draw (0,0)        node (0) [] {}
        -- ++(45:1cm)  node (1) [label=right:$b$] {}
        -- ++(135:1cm) node (2) [label=below:$c$] {}
        -- ++(225:1cm) node (3) [label=left:$a$] {}
        -- ++(90:1cm)  node (4) [label=left:$d$] {}
        -- ++(45:1cm)  node (5) [] {}
        -- ++(315:1cm) node (6) [label=right:$e$] {}
        -- ++(1);


        \draw (2) -- (5);
        \draw (0) -- (3);
\end{tikzpicture}

\end{center}
\caption{\label{aa} A finite non-distributive meet-complemented lattice without $\Box$}
\end{figure}
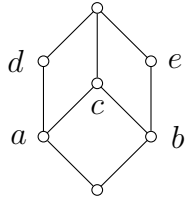

There are finite meet-complemented lattices where $\Box$ does not exist for some element.
For instance, $\Box c$ does not exist in the seven element non-distributive lattice given in Figure \ref{aa}. 
However, $\Box$ exists in the pentagon, where, taking the atom $a$,
it can be seen that the T-property $\Box a \leq a$ is not the case.


We do not have the S4-Property $\Box \leq \Box \Box$,
which can be seen in the lattice $R_8$ (see Figure \ref{rn8}) taking the coatom $r_c$.

As maybe expected, we do not have $\Box(a \vee b) \leq \Box a \vee \Box b$
(consider the coatoms in the lattice $1\oplus 2^{2}$).


Let us use the notation $\mathbb{ML^\Box}$ for the class of meet-complemented lattices expanded with necessity.

It is natural to ask whether $\mathbb{ML^\Box}$ is an equational class.
This we answer positively, as it is enough to add,
to any set of identities proving that $\mathbb{ML}$ is an equational class, the following ones:

\smallskip

{\bf{($\Box$E)}} $x \vee \neg \Box x \approx 1$,

\smallskip

{\bf{($\Box$I1)}} $\Box1 \approx 1$, 

\smallskip

{\bf{($\Box$I2)}} $\Box(x \vee \neg y) \wedge y \approx \Box x \wedge y$.

\noindent Proof: From the given equations we have to prove ($\Box$I) using equational logic, that is, 
supposing $x \vee \neg y \approx 1$ we must get to $\Box x \wedge y \approx y$. 
Suppose $x \vee \neg y \approx 1$. 
Then, using ($\Box$I1), we have $\Box(x \vee \neg y) \approx 1$. 
So, $\Box(x \vee \neg y) \wedge y \approx 1 \wedge y \approx y$. 
Now, using ($\Box$I2), it follows that $\Box x \wedge y \approx y$, as desired.   

\begin{rem}
 In fact, instead of \emph{($\Box$E)}, \emph{($\Box$I1)}, and \emph{($\Box$I2)},
 we may just take $y \preccurlyeq x \vee \neg \Box x$, $x \preccurlyeq \Box \neg (x \wedge \neg x)$,
 and $\Box (x \vee \neg y) \wedge \neg \neg y \preccurlyeq \Box x$, respectively, where, in general, 
 $x \preccurlyeq y$ abbreviates $x \wedge y \approx x$. 
\end{rem}



Let us define, in the context of a join semi-lattice {\bf{L}}, 
the dual of intuitionistic negation $Da =$ max$\{b \in L: a \vee b \leq c$, for all $c \in L \}$. 
It follows that $\bf{L}$ will be bounded with top $= 1$. 
And we will have $Da \leq b$ iff $a \vee b=1$, for any $a, b \in L$. 

Now, let us compare $\Box$ with $\neg D$.

\begin{prop} \label{DgivesBox}
 Let $\bf{A} \in \mathbb{ML}$.
 If $D$ exists in $A$, then $\Box$ also exists in $A$ with $\Box = \neg D$.
\end{prop}

\begin{proof}
 Firstly, as $a \vee Da=1$, we have $a \vee \neg \neg Da=1$.
 Secondly, suppose $a \vee \neg b=1$. Then, $Da \leq \neg b$. So, $b \leq \neg Da$.
\end{proof}

\begin{cor}
 If both $D$ and $\Box$ exist in a meet-complemented lattice, then $\Box = \neg D$.
\end{cor}

On the other hand, the six-element meet-complemented lattice that results from adding a new bottom to the diamond
is an example of a finite lattice with $\Box$ where $D$ does not exist, which can be seen considering any of the coatoms.
Also, the infinite lattice $1\oplus({\bf{\mathbb{N}}} \times {\bf{\mathbb{N}}} )^{\partial}$ is a lattice with $\Box$ such that
$D$ does not exist, where $\mathbb{N}$ denotes the set of natural numbers including $0$
and the exponent $\partial$ indicates the operation of `turning upside down' the argument (see Figure \ref{G}). 
We will use $\mathbb{N}^{\ast}$ to denote the set of natural numbers excluding $0$.

\begin{rem}
 In Lemma 2.1.4(6) L\'{o}pez-Escobar gives without proof that ``$DA\equiv \neg \Box A$'' (see \cite[p. 122]{Lop}).
 Note that we do not have $Da=\neg \Box a$ in the lattice $1 \oplus 2^2$ taking $a$ to be any coatom.
\end{rem}

\begin{prop}
Let {\bf{A}} $\in \mathbb{ML^\Box}$.
 If $D$ exists in $A$, then $D \leq \neg \Box$.
\end{prop}

\begin{proof}
 It is enough to prove that $\neg \Box \in \{b \in A: a \vee b=1 \}$, which is exactly what ($\Box$E) says.
\end{proof}

Regarding Boolean elements, we have the following fact.

\begin{lem} \label{lBB}
 Let $\bf{A} \in \mathbb{ML^\Box}$ and $a \in A$.
 Then, $a$ is Boolean iff $a \leq \Box a$.
\end{lem}

\begin{proof}
 It follows immediately using ($\Box$I) and ($\Box$E).
\end{proof}

Note that also the following holds.

\begin{lem}
 Let {\bf{A}} $\in \mathbb{ML^\Box}$ and $a \in A$.
 Then, $a \leq \Box a$ iff $a \leq \Box^{n}a$, for any $n \in \mathbb{N^{\ast}}$.
\end{lem}

\begin{proof}
 It follows by induction using $\Box$-monotonicity and $\leq$-transitivity.
\end{proof}

Now, let us consider the following operation concerning Boolean elements.

\begin{defi}
Let $\bf{A} \in \mathbb{ML^\Box}$ and $a \in A$.
Then, the greatest Boolean element below $a$ is $Ba = max \{b \in A: b \leq a$ and $b \vee \neg b = 1 \}$.
\end{defi}

We have, for any $a, b \in A$, {\bf{(BE1)}}: $Ba \leq a$, {\bf{(BE2)}}: $Ba \vee \neg Ba = 1$, and
{\bf{(BI)}}: if $b \leq a$ and $b \vee \neg b=1$, then $b \leq Ba$.

Now, let us compare $\Box$ with $B$.

\begin{prop}
 Let {\bf{A}} $\in \mathbb{ML^\Box}$.
 If $B$ exists in $A$, then $B \leq \Box$.
\end{prop}

\begin{proof}
 It is enough to see that $a \vee \neg Ba=1$, which follows from (BE2), (BE1), and Corollary \ref{cL1}.
\end{proof}

The reciprocal inequality is not the case.
To see it, take $R_8$, where the only Boolean elements are $0$ and $1$, and observe that $\Box r_c=r_a \nleq 0 = B r_c$.


\vspace{5pt}

In the rest of this section we prove that the existence of $\Box$ does not imply the existence of $B$.

\begin{defi}
 Let $P$ be a partial order.
 We will say that an upset $U$ of $P$ is Boolean iff $(\downarrow U)^{c} \cup U = P$.
\end{defi}

\begin{lem}
 An upset $U$ of a partial order $P$ is Boolean iff $U$ is decreasing.
\end{lem}

Let us consider the set $(\mathbb{N}, \preccurlyeq)$ of natural numbers with the unusual order $\preccurlyeq$ as in Figure \ref{Ndo}.

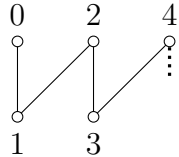
\begin{figure} [ht]
\begin{center}

\begin{tikzpicture}

    \tikzstyle{every node}=[draw, circle, fill=white, minimum size=4pt, inner sep=0pt, label distance=1mm]

    \draw (0,0)

           ++(90:.5cm)		node (4) [label=above:$4$] {}
        -- ++(225:1.4142cm)	node (3) [label=below:$3$] {}
        -- ++(90:1cm)		node (2) [label=above:$2$] {}
        -- ++(225:1.4142cm)	node (1) [label=below:$1$] {}
        -- ++(90:1cm)		node (0) [label=above:$0$] {};

    \draw [dotted, very thick] (4)--(0,0);

\end{tikzpicture}

\end{center}
\caption{\label{Ndo} The set of natural numbers with the given order order}
\end{figure}

\noindent Let $X$ be the poset $\coprod\limits_{i \in \mathbb{N}}\mathbb{N}$, that is,
$X$ consists of enumerable non-comparable copies of $\mathbb{N}$.

\noindent We may also think of $X$ as defined by $\{(i,j): i, j \in N \}$,
where $(i,j) \leq_{X} (i', j')$ iff $i=i'$ and $j \preccurlyeq j'$.
Let us write \\

$X_{i} = \{(i,j) \in X: j \in N \}$, for each $i \in N$, and

$U_{i}=U \cap X_{i}$, for any $U \subseteq X$. \\


\noindent The set $X_i$ will be called the $i$th component of $X$ and the set $U_i$ the $i$th component of $U$.

\begin{defi}
 A distinguished set of X is a subset $U \subseteq X$ such that 
 there is a finite set of indices $I \subseteq \mathbb{N}$ such that \\

 \begin{tabular}{lll}
 \emph{(a)} & $U_i$ is finite & if $i \in I$, \\

 \emph{(b)} & $U_i$ is cofinite relative to $X_i$ & if $i \notin I$.
 \end{tabular}

\vspace{5pt}
 
\noindent A codistinguished set of X is a subset $U \subseteq X$ such that $U^c$ is a distinguished set.
\end{defi}

\noindent That is, a distinguished set has finite components for finite indices and cofinite components for cofinite indices. 
On the other hand, a codistinguished set has cofinite components for finite indices and finite components for cofinite indices.

The following Facts follow straightforwardly from the definition of (co)distinguished set, and they supply a proof for Theorem 1.

\noindent Fact 1. $X$ is distinguished and $\emptyset$ is codistinguished.

\noindent Fact 2. The union and intersection of two distinguished sets are distinguished sets.

\noindent Corollary of Fact 2. The union and intersection of two codistinguished sets are codistinguished sets.

\noindent Fact 3. If $U$ is a distinguished set and $V$ is a codistinguished set, then $U \cap V$ is a codistinguished set.

\noindent Corollary of Fact 3. If $U$ is a distinguished set and $V$ is a codistinguished set, 
then $U \cup V$ is a distinguished set.

\begin{teo}
 The set of distinguished sets and codistinguished sets with the operations $\cap$, $\cup$, and $^c$
 forms a Boolean subalgebra $A$ of the subset-algebra of $X$.
\end{teo}

Let us now take the sublattice $L$ of the algebra $A$ in the theorem, 
considering only the increasing distinguished sets and codistinguished sets.
It should be clear that this lattice is not closed for $^c$.
However, we have meet-complement $\neg$ and join-complement $D$ defining

\begin{tabular}{l}
$\neg U := (\downarrow U)^c$, \\ 

$DU := \uparrow (U^c)$. \\ 
\end{tabular}

\noindent Note that using Proposition \ref{DgivesBox} we have that $\Box$ exists in $L$.

\noindent Let us now see that $B$ does not exist for the increasing distinguished set $V$ defined by

\vspace{5pt}

 \begin{tabular}{lll}
 (a) & $V_i =(i,j)$, for all $j \in \mathbb{N}$ & if $i$ is even, \\

 (b) & $V_i =(i,j)$, for all $j \in \mathbb{N}-\{ 0,1 \}$ & if $i$ is odd. \\
 \end{tabular}

\vspace{5pt}

\noindent Now, as we have already seen, a subset of $X$ is Boolean iff it is decreasing.
So, any component of a Boolean element of $L$ will be both increasing and decreasing.
Now, it is easy to see that the odd components of $V$ are not decreasing, as the element $1$ is excluded. 
So, $V$ is not Boolean.
Which are the Boolean elements below $V$?
As the components of a Boolean element have to be both increasing and decreasing, 
they can only be the $\emptyset$ or $\mathbb{N}$.
In the case of the odd components, we are forced to put the $\emptyset$.
Now, in the case of the even components, 
if we choose a non-zero finite number of them to be $\mathbb{N}$ and the rest to be $\emptyset$,
we get a codistinguished sets. 
And we can do that in an infinite number of ways, obtaining a set Boolean elements that has no maximum.
So, $BV$ does not exist.

So, the given example proves that a meet-complemented lattice with $D$ need not have $B$.

\section{Meet-complemented lattices with possibility}

Let us consider a meet-complemented lattice ${\bf{A}}$ where there exists possibility defined as
$\Diamond a = min\{b \in A: \neg a \vee b = 1 \}$, for any $a \in A$.
It is equivalent to state both

\smallskip

{\bf{($\Diamond$I)}} $\neg a \vee \Diamond a = 1$ and

\smallskip

{\bf{($\Diamond$E)}} if $\neg a \vee b = 1$, then $\Diamond a \leq b$.

\begin{rem}
 Note that $\Diamond \emph{I}$ is equivalent to the reciprocal of $\Diamond \emph{E}$.
 So, $\Diamond \emph{I}$ and $\Diamond \emph{E}$ taken together are equivalent to saying that
 $\Diamond a \leq b$ iff $\neg a \vee b = 1$.
\end{rem}

Let us use the notation ${\mathbb{ML^\Diamond}}$ for the class of meet-complemented lattices with $\Diamond$.

\begin{lem}
Let $\bf{A} \in \mathbb{ML^\Diamond}$. Then, we have

 \noindent $\Diamond$-Monotonicity: if $a \leq b$, then $\Diamond a \leq \Diamond b$, for all $a, b \in A$.
\end{lem}

\begin{proof}
 By ($\Diamond$I) we have $\neg b \vee \Diamond b = 1$. 
 From $a \leq b$, it follows that $\neg b \leq \neg a$. 
 So, $\neg a \vee \Diamond b = 1$, which, using ($\Diamond$E), gives $\Diamond a \leq \Diamond b$. 
\end{proof}

Using $\Diamond$-monotonicity and $\neg$-antimonotonicity we immediately get that
if $x \leq y$, then $\neg \Diamond y \leq \neg \Diamond x$,
which may be considered as an algebraic form of the S3-Schema in modal logic.

As in the case of $\Box$, we will be interested in modalities.
First note that we have the following equality.

\begin{lem} \label{D(DN)}
 Let us consider a lattice of $\mathbb{ML^\Diamond}$. Then,
 $\Diamond \neg \neg = \Diamond$.
\end{lem}

\begin{proof}
 By ($\Diamond$I), $\neg a \vee \Diamond a = 1$.
 So, by (TN) we get $\neg \neg \neg a \vee \Diamond a = 1$.
 Hence, by ($\Diamond$E), $\Diamond \neg \neg \leq \Diamond$.
 The other inequality follows by (DN) and $\Diamond$-monotonicity.
\end{proof}

Regarding modalities of length $\leq 4$ we have the following result.

\begin{lem} \label{dm}
 Let us consider a lattice of $\mathbb{ML^\Diamond}$. Then,
 (i) $\circ \leq \neg \neg$,
 (ii) $\Diamond \leq \neg \neg \Diamond$,
 (iii) $\Diamond \neg \leq \neg \neg \Diamond \neg$,
 (iv) $\neg \Diamond \leq \neg$,
 (v) $\neg \Diamond \neg \leq \neg \neg$,
 (vi) $\neg \leq \neg \neg \Diamond \neg$,
 (vii) $\neg \Diamond \neg \Diamond \leq \neg \neg \Diamond$,
 (viii) $\Diamond \neg \Diamond \neg \leq \Diamond$,
 (ix) $\Diamond \leq \Diamond^{2}$,
 (x) $\Diamond^{2} \leq \Diamond^{3}$,
 (xi) $\Diamond^{3} \leq \Diamond^{4}$,
 (xii) $\Diamond \neg \leq \Diamond \Diamond \neg$,
 (xiii) $\Diamond \neg \Diamond \leq \Diamond \Diamond \neg \Diamond$,
 (xiv) $\neg \Diamond \Diamond \leq \neg \Diamond$,
 (xv) $\neg \Diamond \Diamond \neg \leq \neg \Diamond \neg$,
 (xvi) $\Diamond \neg \Diamond \Diamond \leq \Diamond \neg \Diamond$,
 (xvii) $\neg \neg \Diamond \leq \neg \neg \Diamond \Diamond$,
 (xviii) $\Diamond \neg \Diamond \leq \neg$,
 (xix) $\Diamond \neg \Diamond \Diamond \leq \neg \Diamond$,
 (xx) $\Diamond \Diamond \neg \Diamond \leq \Diamond \neg$,
 (xxi) $\neg \neg \leq \neg \Diamond \neg \Diamond$.
\end{lem}

\begin{proof}
 (i) is (DN).

 \noindent (ii) and (iii) are particular cases of (i) substituting $\Diamond$, $\Diamond \neg$ , and $\Diamond \Diamond \neg$, respectively.

 \noindent (iv) follows by Lemma \ref{lmc}(iii) and ($\Diamond$I).

 \noindent (v) is a particular case of (iv) using $\neg$.

 \noindent (vi) follows from (v) by $\neg$-antimonotonicity and (TN).

 \noindent (vii) is a particular case of (v) using $\Diamond$.

 \noindent (viii) follows by ($\Diamond$I), (vi), and ($\Diamond$E).

 \noindent (ix) follows as, having $\neg \Diamond \vee \Diamond \Diamond = 1$ by ($\Diamond$I),
 we may apply (iv) and ($\Diamond$E).

 \noindent (x), (xi), (xii), and (xiii) are particular cases of (ix) using $\Diamond$, $\Diamond \Diamond$,
 $\neg$, and $\neg \Diamond$, respectively.

 \noindent (xiv) follows from (ix) by $\neg$-antimonotonicity.

 \noindent (xv) is a partocular case of (xiv) using $\neg$.

 \noindent (xvi) follows from (xiv) by $\Diamond$-monotonicity.

 \noindent (xvii) follows from (xiv) by $\neg$-antimonotonicity.

 \noindent (xviii) follows by ($\Diamond$I), (DN), and ($\Diamond$E).

 \noindent (xix) is a particular case of (xviii) using $\Diamond$.

 \noindent (xx) follows from (xviii) by $\Diamond$-monotonicity.

 \noindent (xxi) follows from (xviii) by $\neg$-antimonotonicity.
\end{proof}

\begin{figure} [ht]
\begin{center}

\begin{tikzpicture}

    \tikzstyle{every node}=[draw, circle, fill=white, minimum size=4pt, inner sep=0pt, label distance=1mm]

    \draw (0,0) node (1) [label=right:$\Diamond \Diamond a$] {}
        -- ++(225:1cm) node (2) [] {}
        -- ++(225:1cm) node (b) [label=left:$\Diamond a$] {}
        -- ++(315:1cm) node (4) [] {}
        -- ++(45:1cm) node (5) [] {}
        -- ++(45:1cm) node (6) [] {}
        -- ++(315:1cm) node (na) [label=right:$\Diamond \neg \Diamond a$] {}
        -- ++(225:1cm) node (8) [] {}

        -- ++(270:1cm) node (nb) [label=right:$\neg \Diamond a$] {}
        -- ++(135:1cm) node (10) [] {}
        -- ++(225:1cm) node (a) [label=left:$a$] {}
        -- ++(315:1cm) node (12) [label=right:$\neg \Diamond \Diamond a$] {}
        -- ++(90:1cm) node (13) [] {}
        -- (4);


        \draw (1) -- (6);
        \draw (2) -- (5);
        \draw (8) -- (5);
        \draw (12) -- (nb);
        \draw (4) -- (a);
        \draw (8) -- (13);
        \draw (5) -- (10);

\end{tikzpicture}

\end{center}
\caption{\label{d13} Behaviour of $\Diamond$ and $\neg$ in an atom of $13$}
\end{figure}
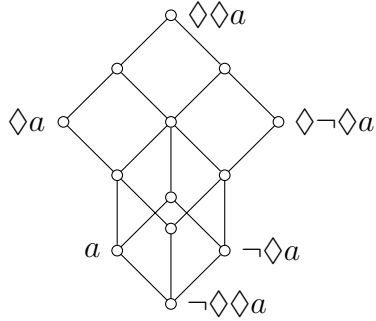

We now give the nodes of the algebras proving that none of the inverse inequalities of Lemma \ref{dm} holds:
For (i): $m$.
For (iv)-(vi), (xviii), (xx), and (xxi): $l$.
For (ii), (ix), and (xii): $r_c$.
For (iii): atom $r_b$ of $R_8$.
For (vii), (viii), (xiii), (xiv), (xvi), (xvii), and (xix): check the behaviour of $a$ in Figure \ref{d13}.
For (xv): check the behaviour of element labeled $\Diamond a$ in Figure \ref{d13}.
For (x)-(xi) consider representation theory as explained in Section $5$.

We leave to the reader to check the cases where the elements are non-comparable.

\begin{figure} [ht]
\begin{center}

\begin{tikzpicture}

    \tikzstyle{every node}=[draw, circle, fill=white, minimum size=4pt, inner sep=0pt, label distance=1mm]

    \draw (0,0)        node (nDDn) [label=right:$\neg \Diamond \Diamond \neg$] {}
        -- ++(135:1cm) node (nDn)  [label=right:$\neg \Diamond \neg$] {}
        -- ++(135:1cm) node (nn)   [label=left:$\neg \neg$] {}
        -- ++(225:1cm) node (c)    [label=left:$\circ$] {}
        -- ++(45:2cm)  node (nDnD) [label=left:$\neg \Diamond \neg \Diamond$] {}
        -- ++(45:1cm)  node (nnD)  [label=left:$\neg \neg \Diamond$] {}
        -- ++(45:1cm)  node (nnDD) [label=left:$\neg \neg \Diamond \Diamond$] {}
        -- ++(225:1cm) node (nnD)  [] {}
        -- ++(315:1cm) node (D)    [label=right:$\Diamond$] {}
        -- ++(315:1cm) node (DnDn) [label=right:$\Diamond \neg \Diamond \neg$] {}
        -- ++(135:1cm) node (D)    [] {}
        -- ++(45:1cm)  node (DD)   [label=right:$\Diamond^2$] {}
        -- ++(45:1cm)  node (DDD)  [label=right:$\Diamond^3$] {}
        -- ++(45:1cm)  node (DDDD) [label=right:$\Diamond^4$] {};

\end{tikzpicture}

\end{center}
\caption{\label{pmd} Positive modalities for $\neg$ and $\Diamond$ with maximum lenght $=4$}
\end{figure}
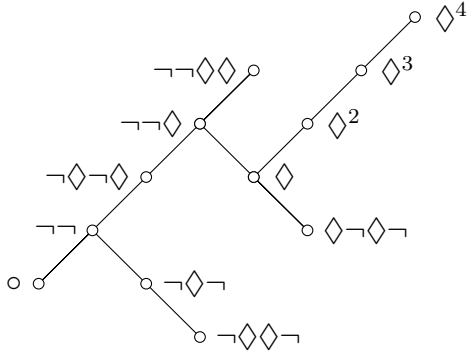

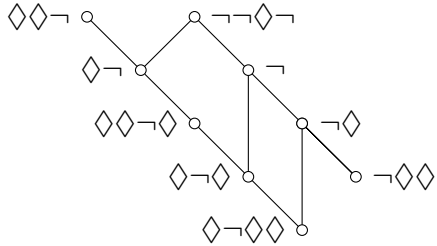
\begin{figure} [ht]
\begin{center}

\begin{tikzpicture}

    \tikzstyle{every node}=[draw, circle, fill=white, minimum size=4pt, inner sep=0pt, label distance=1mm]


    \draw (0,0) node (1) [label=left:$\Diamond \Diamond \neg$] {}
        -- ++(315:1cm) node (2) [label=left:$\Diamond \neg$] {}
        -- ++(45:1cm) node  (3) [label=right:$\neg \neg \Diamond \neg$] {}
        -- ++(315:1cm) node (4) [label=right:$\neg$] {}
        -- ++(315:1cm) node (5) [label=right:$\neg \Diamond$] {}
        -- ++(315:1cm) node (6) [label=right:$\neg \Diamond \Diamond$] {}
        -- ++(135:1cm) node (5a) [] {}

        -- ++(270:1.4142cm) node (7) [label=left:$\Diamond \neg \Diamond \Diamond$] {}
        -- ++(135:1cm) node (8) [label=left:$\Diamond \neg \Diamond$] {}
        -- ++(135:1cm) node (9) [label=left:$\Diamond \Diamond \neg \Diamond$] {}
        -- (2);


        \draw (4) -- (8);

\end{tikzpicture}

\end{center}
\caption{\label{nmd} Negative modalities for $\neg$ and $\Diamond$ with maximum lenght $=4$}
\end{figure}

Regarding the interaction of $\Diamond$ with $\wedge$, $\vee$, $\neg$, $0$, and $1$, we have the following facts.

\begin{prop} \label{D}
 Let $\bf{A} \in \mathbb{ML^\Diamond}$ and let $a, b \in A$. Then,
 (i) $\Diamond (a \wedge b) \leq \Diamond a \wedge \Diamond b$,
 (ii) $\Diamond a \vee \Diamond b \leq \Diamond (a \vee b)$,
 (iii) $\Diamond \neg a \vee \Diamond \neg b \leq \Diamond \neg (a \wedge b)$,
 (iv) $\Diamond \neg (a \wedge b) = \Diamond ( \neg a \vee \neg b)$,
 (v) $\Diamond a = 0$ iff $a = 0$,
 (vi) $\Diamond 1 = 1$, 
 (vii) $\neg \Diamond 0 = 1$. 
\end{prop}

\begin{proof}
 (i) follows from $a \wedge b \leq a, b$ using $\Diamond$-monotonicity and ($\wedge$I).

 \noindent (ii) follows from $a, b \leq a \vee b$ using $\Diamond$-monotonicity and ($\vee$E).

 \noindent (iii) By $\neg$-antimonotonicity, we have both $\neg a \leq \neg (a \wedge b)$
 and $\neg b \leq \neg(a \wedge b)$.
 So, by $\Diamond$-monotonicity, we get both
 $\Diamond \neg a \leq \Diamond \neg (a \wedge b)$ and $\Diamond \neg b \leq \Diamond \neg (a \wedge b)$.
 Then, by a property of $\vee$, we finally get $\Diamond \neg a \vee \Diamond \neg b \leq \Diamond \neg (a \wedge b)$.

 \noindent (iv) We have $\neg (a \wedge b) \leq \neg \neg (\neg a \vee \neg b)$.
 Using $\Diamond$-monotonicity, we get $\Diamond \neg (a \wedge b) \leq \Diamond \neg \neg (\neg a \vee \neg b)$.
 So, as $\Diamond \neg \neg = \Diamond$, we get $\Diamond \neg (a \wedge b) \leq \Diamond (\neg a \vee \neg b)$.
 The other inequality follows from $\neg a \vee \neg b \leq \neg (a \wedge b)$ by $\Diamond$-monotonicity.

 \noindent (v) is easy to see using Lemma 1(vii).

 \noindent (vi) and (vii) are immediate.
\end{proof}

\begin{rem}
 Note that something analogous to (iv) in the previous proposition is not the case for $\Box$,
 i.e., we do not have $\Box \neg(a \wedge b)= \Box ( \neg a \vee \neg b)$,
 as can be seen taking the atoms in the lattice $2^{2} \oplus 1$.
 Also, the inequality $\Diamond a \leq \Diamond (a \vee b)$
 is useful to prove that monotonicity of $\Diamond$ is given by an identity.
\end{rem}

\begin{figure} [ht]
\begin{center}

\begin{tikzpicture}

  \tikzstyle{every node}=[draw, circle, fill=white, minimum size=4pt, inner sep=0pt, label distance=1mm]

    \draw (0,0)

           ++(45:.5cm)	node (2) [label=right:$a$] {}
        -- ++(135:1cm)	node (1) [] {}
        -- ++(225:1cm)	node (na) [label=left:$\neg a$] {}
        -- ++(315:1cm)	node (0) [] {};

    \draw [dotted, very thick] (2)--(0,0);

\end{tikzpicture}

\end{center}
\caption{\label{lwD}A meet-complemented lattice where $\Diamond$ does not exist}
\end{figure}
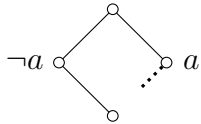

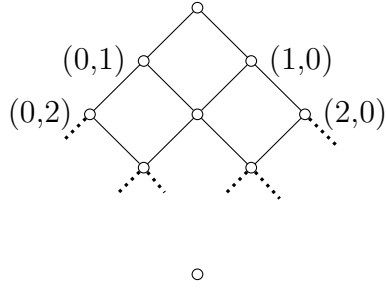
\begin{figure} [ht]
\begin{center}

\begin{tikzpicture}

  \tikzstyle{every node}=[draw, circle, fill=white, minimum size=4pt, inner sep=0pt, label distance=1mm]

    \draw (0,0)

	   ++(45:.5cm)	node (02) [label=left:(0{,}2)] {}
        -- ++(45:1)	node (01) [label=left:(0{,}1)] {}
        -- ++(45:1)	node (00) [] {}
        -- ++(315:1)	node (10) [label=right:(1{,}0)] {}
        -- ++(315:1)	node (20) [label=right:(2{,}0)] {}
        -- ++(225:1)	node (21) [] {}
        -- ++(135:1)	node (11) [] {}
        -- ++(225:1)	node (12) [] {};

    \draw (02)--(12);
    \draw (01)--(11);
    \draw (10)--(11);

    \draw [dotted, very thick] (02)--(0,0);
    \draw [dotted, very thick] (12)--(315:1);
    \draw [dotted, very thick] (12)--(333:1.5);
    \draw [dotted, very thick] (21)--(342:2.25);
    \draw [dotted, very thick] (21)--(345:2.95);
    \draw [dotted, very thick] (20)--(359:3.6);

    \draw (315:2.5) node [] {};

\end{tikzpicture}

\end{center}
\caption{\label{G} A G\"{o}del algebra where $\Diamond$ does not exist}
\end{figure}

There are meet-complemented lattices where $\Diamond$ does not exist for some element.
For instance, $\Diamond a$ does not exist in the non-distributive meet-complemented lattice given in Figure \ref{lwD},
where the reader has to imagine the existence of denumerable elements downwards from node $a$.
Also, $\Diamond$ does not exist in the G\"{o}del algebra given in Figure \ref{G}.

In particular, observe that $\Diamond$ exists in the pentagon, where, taking the coatom that is not an atom,
it can be seen that the T-property $a \leq \Diamond a$ is not the case
(taking the atom-coatom it is also neither the case that $\neg \leq \Diamond \neg$ nor that $\neg \Diamond \leq \Diamond \neg$).
We also do not have the S4-Property $\Diamond \Diamond \leq \Diamond$,
which can be seen in the lattice $R_8$ (see Figure \ref{rn8}) taking the atom $r_a$.

We have seen that we have $\Diamond a \vee \Diamond b \leq \Diamond (a \vee b)$.
In the lattice in Figure \ref{aa}, where $\Diamond$ does exist,
it may be seen that the reciprocal does not hold, as we have that
$1=\Diamond (a \vee b) \nleq \Diamond a \vee \Diamond b=a \vee b=c$.
In the same lattice it may be seen that the reciprocal of
$\Diamond \neg a \vee \Diamond \neg b \leq \Diamond \neg (a \wedge b)$, proved in Proposition \ref{D}(iii),
does not hold, as we have that
$1 = \Diamond \neg (d \wedge e) \nleq \Diamond \neg d \vee \Diamond \neg e = \Diamond e \vee \Diamond d = b \vee a = c$.
Again in the same lattice, taking, for instance,  the coatom $d$,
it may be seen that it is not the case that $\neg \Diamond \leq \Diamond \neg$,
as $\neg \Diamond d = \neg a = e \nleq b = \Diamond e = \Diamond \neg d$.

Taking the meet-irreducible atom $r_a$ in $R_8$, it may be seen that it is not the case that
$\neg \neg \Diamond \leq \Diamond$.

Also, by $\Diamond$-monotonicity, we have $\Diamond (a \wedge b) \leq \Diamond a \wedge \Diamond b$,
which may be considered as a form of the well-known S2-schema in modal logic.
However, the reciprocal is not the case considering the lattice $2^{2}\oplus1$.
The same example shows that neither $\Diamond a \wedge \Diamond \neg a \leq \Diamond (a \wedge \neg a)$ nor
$\Diamond a \wedge \Diamond \neg a = 0$ are the case.

Let us use the notation $\mathbb{ML^\Diamond}$ for the class of meet-complemented lattices expanded with possibility.

It is natural to ask whether $\mathbb{ML^\Diamond}$ is an equational class.
Let us see that it is not even a quasi-equational class considering the lattice $R_8$ (see Figure \ref{rn8}).
Indeed, consider the six-element lattice {\bf{L}} that results from $R_8$ indentifying $e$ with the top $1$ and $d$ with $r_c$.
Then, define the obvious homomorphism $h$ from $R_8$ to {\bf{L}}.
It is easily seen that $\Diamond$ is not preserved under $h$, as, for instance, $h\Diamond c=h1 \neq hc = \Diamond hc$.


Now, let us compare $\Diamond$ with $D\neg$.

\begin{prop}
 Let $\bf{A} \in \mathbb{ML^\Diamond}$.
 If $D$ exists in $A$, then $\Diamond$ also exists in $A$ with $\Diamond = D\neg$.
\end{prop}

\begin{proof}
 Firstly, we have $\neg a \vee D \neg a = 1$.
 Secondly, suppose $\neg a \vee b = 1$. Then, $D \neg a \leq b$.
\end{proof}

On the other hand, the infinite lattice $1 \oplus (\mathbb{N} \times \mathbb{N})^{\partial}$ in Figure \ref{G}
is a lattice with $\Diamond$ such that $D$ does not exist.

\begin{prop}
 Let $\bf{A} \in \mathbb{ML^\Diamond}$. Then,
 if $D$ exists in $A$, then $\Diamond \neg \leq D$.
\end{prop}

\begin{proof}
 It is enough to prove that if $a \vee b = 1$, then $\Diamond \neg a \leq b$.
 Now, from $a \vee b = 1$, it follows that $\neg \neg a \vee b = 1$.
 Using ($\Diamond$E), it follows that $\Diamond \neg a \leq b$.
\end{proof}

Reciprocally to what happens with Boolean elements and $\Box$, we have the following fact.

\begin{lem} \label{lDB}
 Let $\bf{A} \in \mathbb{ML^\Diamond}$ and $a \in A$.
 Then, $a$ is Boolean iff $\Diamond a \leq a$.
\end{lem}

\begin{proof}
 It follows immediately using ($\Diamond$I) and ($\Diamond$E).
\end{proof}

Regarding $B$, we may say that it exists in the lattice in Figure \ref{lwD},
though $\Diamond$ does not exist at node $a$, as already said.

We may also consider the dual of $B$.
It may be seen that such operation may be defined as $\neg B \neg$.
It may be easily seen that if $B$ exists in a meet-complemented lattice with $\Diamond$,
we have that $\Diamond \leq \neg B \neg$.

\section{Necessity and possibility together}

Let us consider a meet-complemented lattice with necessity as in Section 2 and add possibility,
as defined in the previous section.
Let us use the notation $\mathbb{ML^{\Box \Diamond}}$
for the class $\mathbb{ML^\Box}$ expanded with possibility.

Some properties that involve both modal operators are the following, where we use ``B'' for Brouwerian and ``A'' for Adjunction.
By the way, so called ``Brouwerian'' properties seem to have originated by Becker in 1930 (see \cite{Beck}).

\begin{lem} \label{BD}
 Let {\bf{A}}$\in \mathbb{ML^{\Box \Diamond}}$ and let $a, b \in A$. Then,

 (i) $\neg \neg \leq \Box \Diamond$,

 (ii) $\Box \leq \Box \Diamond$,

 {\bf{(B1)}} $\circ \leq \Box \Diamond$,

 {\bf{(B2)}} $\Diamond \Box \leq \circ$,

 {\bf{(A)}} $\Diamond a \leq b$ iff $a \leq \Box b$,

 (iii) $\Diamond \Box \Diamond = \Diamond$ and $\Box \Diamond \Box = \Box$.
\end{lem}

\begin{proof}
 (i) follows using ($\Diamond$I) and ($\Box$I).

 \noindent (ii) follows from (i), Lemma \ref{lmcln}(i), and transitivity of $\leq$.

 \noindent (B1) follows from (i).

 \noindent (B2) follows using ($\Box$E) and ($\Diamond$E).

 \noindent (A) $\Rightarrow$) follows by $\Box$-monotonicity, (iii), and $\leq$-transitivity.
 $\Leftarrow$) follows by $\Diamond$-monotonicity, (iv), and $\leq$-transitivity.

 \noindent (iii) The inequality $\Diamond \Box \Diamond \leq \Diamond$ holds because of (B2).
 For the other inequality apply monotonicity of $\Diamond$ to (B1).
 The other equality is proved analogously.
\end{proof}

We also have the following facts involving negation.

\begin{prop} \label{NBD}
 Let {\bf{A}}$\in \mathbb{ML^{\Box \Diamond}}$. Then,

 \noindent (i) $\Diamond \leq \neg \Box \neg$,
 (ii) $\neg \Diamond = \Box \neg$,
 (iii) $\neg \neg \Diamond = \neg \Box \neg$,
 (iv) $\Box \neg \neg = \neg \Diamond \neg$,
 (v) $\Diamond \neg \leq \neg \Box$,
 (vi) $\Box \leq \neg \Diamond \neg$,
 (vii) $\neg \Diamond \neg \leq \Box \Diamond$,
 (viii) $\Box \neg \leq \neg \Box \Diamond$,
 (ix) $\Box \Diamond \leq \neg \Box \neg$,
 (x) $\Diamond \neg \Box \neg = \Diamond \Diamond$.
\end{prop}

\begin{proof}
 (i) follows by ($\Diamond$E) from $\neg \Box \neg \vee \neg = 1$, which holds by ($\Box$E).

 \noindent (ii) The inequality $\neg \Diamond \leq \Box \neg$ follows by ($\Box$I) from
 $\neg \vee \neg \neg \Diamond = 1$, which follows from ($\Diamond$I).
 The other inequality of (ii) follows from (i) using $\neg$-antimonotonicity and the fact that $\neg \neg \Box = \Box$.

 \noindent (iii) and (iv) follow immediately from (ii).

 \noindent (v) follows by ($\Diamond$E) from $\neg \Box \vee \neg \neg=1$, which follows from ($\Box$E).

 \noindent (vi) follows from (v) using $\neg$-antimonotonicity and $\neg \neg \Box = \Box$.

 \noindent (vii) follows, using ($\Box$I), from $\Diamond \vee \neg \neg \Diamond \neg = 1$,
 which holds because, due to ($\Diamond$I), $\Diamond \vee \neg = 1$ and $\neg \leq \neg \neg \Diamond \neg$,
 as seen in Lemma \ref{D}(ii).

 \noindent (viii) Using Lemma 2(i) we get $\Box \Diamond \leq \neg \neg \Diamond$.
 So, by $\neg$-antimonotonicity and (TN) we get $\neg \Diamond \leq \neg \Box \Diamond$. Finally, use (ii).

 \noindent (ix) follows applying $\neg$-antimonotonicity to (vi) and then using that $\neg \neg \Box = \Box$,
 which holds due to Lemma \ref{lmcln}(v).

 \noindent (x) Use (ii) and $\Diamond \neg \neg = \Diamond$, due to Lemma 3(vi).
\end{proof}

Regarding the inequalities not considered in the previous proposition,
taking the middle element of the three-element chain it may be seen that
$\neg \Box \leq \neg \neg \Diamond \neg$ is not the case.
So, neither $\neg \Box \leq \Diamond \neg$ nor $\neg \Diamond \neg \leq \neg \neg \Box$ are the case,
the last of which implies that also $\neg \Diamond \neg \leq \Box$ is not the case.
Also, taking the non-atom coatom of the pentagon, it may be seen that
neither $\neg \Box \neg \leq \Diamond$ nor $\neg \neg \Diamond \leq \Diamond$ are the case.

\begin{prop} \label{bWdV}
Let $\bf{A} \in \mathbb{ML}$ and $\{ a_i : i \in I \}$ an arbitrary subset of $A$.
Then,

\noindent(i) if both $\bigwedge_{i \in I} a_i$ and $\bigwedge_{i \in I} \Box a_i$ exist,
then $\bigwedge_{i \in I} \Box a_i = \Box \bigwedge_{i \in I} a_i$ and

\noindent (ii) if both $\bigvee_{i \in I} a_i$ and $\bigvee_{i \in I} \Box a_i$ exist,
then $\bigvee_{i \in I} \Diamond a_i = \Diamond \bigvee_{i \in I} a_i$.
\end{prop}

\begin{proof}
 (i) Using that $\bigwedge_{i \in I} a_i \leq a_i$, for every $i \in I$, and $\Box$-monotonicity,
 it follows that $\Box \bigwedge_{i \in I} a_i \leq \Box a_i$, for every $i \in I$.
 So, $\Box \bigwedge_{i \in I} a_i \leq \bigwedge_{i \in I} \Box a_i$.
 For the other inequality, we have $\bigwedge_{i \in I} \Box a_i \leq \Box a_i$, for every $i \in I$.
 Using (A), it follows that $\Diamond \bigwedge_{i \in I}\Box a_i \leq a_i$, for every $i \in I$.
 Then, $\Diamond \bigwedge_{i \in I}\Box a_i \leq \bigwedge_{i \in I} a_i$.
 Using (A) in the other direction, we get $\bigwedge_{i \in I} \Box a_i \leq \Box \bigwedge_{i \in I} a_i$.

 \noindent (ii) Dualize previous proof.
\end{proof}

\begin{cor} \label{bwdv}
 Let {\bf{A}}$\in \mathbb{ML^{\Box \Diamond}}$ and let $a, b \in A$. Then,

 (i) $\Box (a \wedge b) = \Box a \wedge \Box b$ and (ii) $\Diamond (a \vee b) = \Diamond a \vee \Diamond b$.
\end{cor}

\begin{rem}
 In Section 3 we saw that (ii) of the just given corollary does not hold in $\mathbb{ML^\Diamond}$.
 Now we see that it does hold in $\mathbb{ML^{\Box \Diamond}}$.
 So, borrowing terminology from Logic,
 we have that $\mathbb{ML^{\Box \Diamond}}$ is not a conservative expansion of $\mathbb{ML^\Diamond}$.
\end{rem}

The following lemma will be useful. 

\begin{lem} \label{TDE}
 Both $\leq$-transitivity and $\Diamond$-monotonicity may be given by equations. 
\end{lem}

\begin{proof}
 The case of $\leq$-transitivity is easy to check. 
 For $\Diamond$-monotonicity, note that we have the equation 
 
 {\bf{($\Diamond$E1)}} $\Diamond x \vee \Diamond (x \vee y)=\Diamond (x \vee y)$. 
 
 \noindent Now, suppose $x \vee y \approx y$. 
 Then, $\Diamond (x \vee y) \approx \Diamond y$. 
 Hence, substituting $\Diamond y$ for $\Diamond (x \vee y)$ in ($\Diamond$E1), 
 we get $\Diamond x \vee \Diamond y = \Diamond y$. 
\end{proof}

It is natural to ask whether $\mathbb{ML^{\Box \Diamond}}$ is an equational class.
The answer is positive, taking the identities for $\mathbb{ML^\Box}$ and also the following, 
where as before $x \preccurlyeq y$ abbreviates $x \wedge y \approx x$, 

\smallskip

{\bf{($\Diamond$I)}} $\neg x \vee \Diamond x \approx 1$,

\smallskip

{\bf{($\Diamond$E1)}} $\Diamond x \preccurlyeq \Diamond (x \vee y)$, 

\smallskip

{\bf{($\Diamond$E2)}} $\Diamond \Box x \preccurlyeq x$.

\smallskip

\noindent Proof: From the given equations we have to prove ($\Diamond$E) using equational logic. 
So, let us suppose $\neg x \vee y \approx 1$. 
By equation ($\Box$I1) it follows that $\Box(\neg x \vee y) \approx 1$,
which implies $\Box (\neg x \vee y) \wedge x \approx x$, which, using equation ($\Box$I2),
implies $\Box y \wedge x \approx x$, i.e. $x \preccurlyeq \Box y$, which, using $\Diamond$-monotonicity, 
which may be given by equations as proved in Lemma \ref{TDE}, gives $\Diamond x \preccurlyeq \Diamond \Box y$. 
Finally, from the just given inequality and using equation ($\Diamond$E2) and $\leq$-transitivity, 
we get $\Diamond x \preccurlyeq y$.

\begin{rem}
 It is equivalent to use $x \preccurlyeq \Box \Diamond x$ instead of $\Diamond \emph{I}$.
\end{rem}


Using results of the previous two sections, we have the following fact.

\begin{prop}
 Let {\bf{A}}$\in \mathbb{ML^{\Box \Diamond}}$.
 Then, if $D$ exists in $A$, we have $\Diamond \neg \leq D \leq \neg \Box$.
\end{prop}

Regarding $B$ we have the following result.

\begin{prop}
Let {\bf{A}}$\in \mathbb{ML^{\Box \Diamond}}$.
If $\bigwedge \{\Box^{n}a: n \in \mathbb{N} \}$ exists for all $a \in A$,
then $Ba = \bigwedge \{\Box^{n}a: n \in {\mathbb{N}} \}$, for all $a \in A$.
\end{prop}

\begin{proof}
 We prove, for any $a \in A$, (i) $\bigwedge \{\Box^{n}a: n \in {\mathbb{N}} \} \leq a$,
 (ii) $\bigwedge \{\Box^{n}a: n \in {\mathbb{N}} \} \vee \neg  \bigwedge \{\Box^{n}a: n \in {\mathbb{N}} \} = 1$, and
 (iii) for any $b \in A$, if $b \leq a$ and $b \vee \neg b = 1$, then $b \leq \bigwedge \{\Box^{n}a: n \in {\mathbb{N}} \}$.

 \noindent (i) Just note that $\Box^{0}a = a$.

 \noindent (ii) $\bigwedge \{ \Box^{n}a: n \in \mathbb{N} \} \leq \Box^{n}a$, for all $n \in \mathbb{N}$.
 So, $\bigwedge \{ \Box^{n}a: n \in \mathbb{N} \} \leq \Box^{n}a$, for all $n \in \mathbb{N}^{\ast}$.
 So, $\bigwedge \{ \Box^{n}a: n \in \mathbb{N} \} \leq \Box^{n+1}a$, for all $n \in \mathbb{N}$.
 So, $\bigwedge \{ \Box^{n}a: n \in \mathbb{N} \} \leq \Box \Box^{n}a$, for all $n \in \mathbb{N}$.
 So, by (A), $\Diamond \bigwedge \{ \Box^{n}a: n \in \mathbb{N} \} \leq \Box^{n}a$, for all $n \in \mathbb{N}$.
 So, $\Diamond \bigwedge \{ \Box^{n}a: n \in \mathbb{N} \} \leq \bigwedge \{ \Box^{n}a: n \in \mathbb{N} \}$.
 So finally, using (A) again, $\bigwedge \{ \Box^{n}a: n \in \mathbb{N} \} \leq \Box \bigwedge \{ \Box^{n}a: n \in \mathbb{N} \}$.

 \noindent (iii) Suppose (1) $b \leq a$ and (2) $b \vee \neg b = 1$.
 From (2) using Lemma, it follows (3) $b \leq \Box b$.
 Now, inductively suppose $b \leq \Box^{n}a$, for any $n \in \mathbb{N}$.
 Then, by $\Box$-monotonicity, it follows that $\Box b \leq \Box ^{n+1}a$.
 Using (3) and transitivity, we get $b \leq \Box ^{n+1}a$.
 It follows, by induction, that $b \leq \Box^{n}a$, for any $n \in \mathbb{N}$.
 So, $b \leq \bigwedge \{\Box^{n}a: n \in {\mathbb{N}} \}$.
\end{proof}

\begin{rem}
 Note that in the previous proof we only use the existence of $\Diamond$ in step (ii).
\end{rem}

\section{The distributive extension}
\label{five}

We begin reminding the following facts.

\begin{lem}\label{dist}
 Let {\bf{L}} be a distributive lattice with top $1$.
 Then, for all $a, b, c, d \in L$, if $a \vee b = 1$, $c \wedge a \leq d$, and $c \wedge b \leq d$, then $c \leq d$,
\end{lem}

\begin{proof}
 Suppose $a \vee b = 1$. Then, $(a \vee b) \wedge c = c$.
 Then, using distributivity, $(a \wedge c) \vee (b \wedge c) = c$.
 So, using $c \wedge a \leq d$ and $c \wedge b \leq d$, we have $(c \wedge a) \vee (c \wedge b) \leq d$.
 Then, $c \leq d$.
\end{proof}

\begin{lem}\label{distneg}
 Let {\bf{L}} be a distributive lattice with meet-complement $\neg$ and top $1$ in the signature.
 Then, for all $a,b \in L$,
 (i) $(a \vee b) \wedge \neg b \leq a$,
 (ii) $(a \vee \neg b) \wedge b \leq a$,
 (iii) if $a \vee b = 1$, then $\neg b \leq a$,
 (iv) if $a \vee \neg b = 1$, then $b \leq a$.
\end{lem}

\begin{proof}
 (i) $(a \vee b) \wedge \neg b \leq (a \wedge \neg b) \vee (b \wedge \neg b) \leq a \wedge \neg b \leq a$.
 (ii) is similar.
 (iii) and (iv) follow from (i) and (ii), respectively.
\end{proof}

\smallskip

Next, we consider questions of existence of $\Box$ and $\Diamond$.

\begin{prop}
 Operations $\Box$ and $\Diamond$ exist in every finite meet-complemented distributive lattice.
\end{prop}

\begin{proof}
 We prove the $\Box$-case.
 The $\Diamond$-case is similar.
 Suppose $a \vee \neg b_{1} = 1$ and $a \vee \neg b_{2}=1$.
 Then, $(a \vee \neg b_{1}) \wedge (a \vee \neg b_{2})=1$.
 Now, using distributivity, $(a \vee \neg b_{1}) \wedge (a \vee \neg b_{2}) \leq a \vee (\neg b_{1} \wedge \neg b_{2})$.
 So, $a \vee (\neg b_{1} \wedge \neg b_{2})=1$.
 Using De Morgan, it follows that $a \vee \neg (b_{1} \vee b_{2})=1$.
\end{proof}

\begin{prop}
 There is an (infinite) meet-complemented distributive lattice $\mathbf{A}$ and an element $a\in A$ such that
 $\Box a$ does not exist.
\end{prop}

\begin{proof}
 We owe the following example to Franco Montagna (see \cite{AEM}).
 Let $[0,1]_{G}$ be the standard G\"{o}del algebra on $[0,1]$ and let us consider 

 $A_{1} = \{ a\in [0,1]_{G}^{\mathbb{N}}$ such that $\left\{ i\in \mathbb{N}:a_{i}=0\right\}$ is finite $\}$, 
 
 $A_{2} = \{ a\in [0,1]_{G}^{\mathbb{N}}$ such that $\left\{ i\in \mathbb{N}:a_{i}=0\right\} $ is cofinite $\}$, 
 
 $A = A_{1}\cup A_{2}$. 
 
 \noindent We claim that $A$ is the domain of a subalgebra of $[0,1]_{G}^{\mathbb{N}}$.
 Indeed, if $a,b\in A_{1}$, then $a\wedge b\in A_{1}$, $a\vee b\in A_{1}$, and $a\rightarrow b\in A_{1}$.
 If $a,b\in A_{2}$, then $a\wedge b\in A_{2}$, $a\vee b\in A_{2}$, and $a\rightarrow b\in A_{1}$.
 If $a\in A_{1}$ and $b\in A_{2}$, then $a\wedge b\in A_{2}$, $a\vee b\in A_{1}$, and $a\rightarrow b\in A_{2}$.
 Finally, if $a\in A_{2}$ and $b\in A_{1}$, then $a\wedge b\in A_{2}$, $a\vee b\in A_{1}$, and $a\rightarrow b\in A_{1}$.
 Since $0\in A_{2}$ and $1\in A_{1}$, $A$ is the domain of a subalgebra $\mathbf{A}$ of $[0,1]_{G}^{\mathbb{N}}$.
 Now, let $a$ be such that $a_{i}=1$ if $i$ is even and $a_{i}=\frac{1}{2}$ if $i$ is odd.
 Then, $\{y \in A: x \vee \neg y=1\}$ consists of all elements $a$ such that $a_{i}=0$ for all odd $i$ and
 for all but finitely many even $i$, and $a_{i} \neq 0$ otherwise.
 Clearly, this set has no maximum in $A$.
\end{proof}


We will use the notation $\mathbb{ML_\textrm{d}^{\Box \Diamond}}$
for the class of distributive meet-complemented lattices with both $\Box$ and $\Diamond$.

\begin{lem} \label{(DN)}
 Let $\bf{A} \in {\mathbb{ML_\textrm{d}^{\Box \Diamond}}}$ and let $a, b \in A$. Then,

 \noindent (i) $\Box \leq \circ$, $\circ \leq \Diamond$, and $\Box \leq \Diamond$,
 (ii) $\Box (a \vee \neg a) = \Box a \vee \Box \neg a$,
 (iii) $\Diamond \neg (a \wedge b) = \Diamond \neg a \vee \Diamond \neg b$, and
 (iv) $\Diamond \neg (a \wedge b) = \Diamond (\neg a \vee \neg b)$.
\end{lem}

\begin{proof}
 (i) Using ($\Box$E) and Lemma \ref{distneg}(iv) we get $\Box \leq \circ$.
 It is similar to get $\circ \leq \Diamond$ and $\Box \leq \Diamond$ results by $\leq$-transitivity.

 \noindent (ii) Inequalities $\Box (a \vee \neg a) \wedge a \leq \Box a$ and
 $\Box (a \vee \neg a) \wedge \neg a \leq \Box \neg a$ follow from Proposition \ref{BoxIn}(iii) and (iv), respectively.
 As, using (i) we have $\Box (a \vee \neg a) \leq a \vee \neg a$, the goal follows using distributivity.

 \noindent (iii) By ($\Diamond$I) we have both $\neg \neg a \vee \Diamond \neg a = 1$ and $\neg \neg b \vee \Diamond \neg b = 1$.
 So, $(\neg \neg a \vee \Diamond \neg a) \wedge (\neg \neg b \vee \Diamond \neg b) = 1$.
 So, by distributivity,
 $(\neg \neg a \wedge \neg \neg b) \vee (\neg \neg a \wedge \Diamond \neg b) \vee (\Diamond \neg a \wedge \neg \neg b) \vee (\Diamond \neg a \wedge \Diamond \neg b) = 1$.
 So, by a property of $\vee$, we have that $(\neg \neg a \wedge \neg \neg b) \vee (\Diamond \neg a \vee \Diamond \neg b) = 1$.
 So, by properties of $\wedge$ and $\neg$, it follows that $\neg \neg (a \wedge b) \vee (\Diamond \neg a \vee \Diamond \neg b) = 1$.
 So, by ($\Diamond$E), we have $\Diamond \neg(a \wedge b) \leq \Diamond \neg a \vee \Diamond \neg b$.
 The other inequality was proved in Lemma \ref{D} in Section 3, as it does not require distributivity.

 \noindent (iv) follows from (iii) and Proposition \ref{bwdv}(ii).
\end{proof}

We also have the following generalization of Lemma \ref{(DN)}(i), which is easily seen.

\begin{lem} \label{g(DN)}
Let us consider a lattice of $\mathbb{ML_\textrm{d}^{\Box \Diamond}}$. Then,
$\Box^{n} \leq \circ$, $\circ \leq \Diamond^{n}$, $\Box^{n} \leq \Diamond^{n}$.
\end{lem}

Considering $\neg$ we get $\neg \neg \leq \Diamond$.
Also, $\neg \neg \Diamond \leq \Diamond \Diamond$.

In the distributive extension we still do not have $\neg \Box \neg \leq \Diamond$.
This can be seen in the lattice $R_8$ taking the meet-irreducible atom $r_a$,
where $\neg \Box \neg r_a =1$ and $\Diamond r_a$ is the join-irreducible coatom $r_c$.
We also have that $\neg \neg \Diamond r_a = 1$.
So, neither $\neg \neg \Diamond \leq \Diamond$ is the case.

It is natural to suspect that there will exist an infinite number of modalities, that is,
finite combinations of $\neg$, $\Box$, and $\Diamond$.
In order to see it, let us use some ideas from representation theory.
Let $A$ be an increasing set of a poset $P$.
We have that $\neg A =(\downarrow A)^{c}$, $\Diamond A = \uparrow \downarrow A$, and $\Box A = (\downarrow (\uparrow  A^{c}))^{c}$,
where $\downarrow A := \{p \in P : p \leq a \textrm{ for some } a \in A\}$ and $\uparrow A$ is defined dually.
The characterizations of $\Diamond$ and $\Box$ can also be stated as,

\begin{enumerate}
\item[i.] $x \in \Diamond A$ iff there exist $y, z \in P$ such that $y \leq z, y \leq x$, and $z \in A$, and
\item[ii.] $x \in \Box A$ iff for every $y, z \in P$, if $x \leq y, z \leq y$, then  $z \in A$.
\end{enumerate}

Let us use them to see that there are infinite modalities.
Consider the poset $P_1$ of figure \ref{pp1}.
Then, for $A=\{ a \}$, we have $\Diamond A = P_1$.

\begin{figure} [H]
\begin{center}

\begin{tikzpicture}

    \tikzstyle{every node}=[draw, circle, fill=white, minimum size=4pt, inner sep=0pt, label distance=1mm]

    \draw (0,0) node (1) [label=left:$a$] {}
        -- ++(270:1cm) node (0) [] {}
        -- ++(45:1.4142cm) node (b) [] {};

\end{tikzpicture}

\end{center}
\caption{\label{pp1} A three-element partial order}
\end{figure}
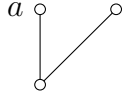




\noindent Consider now the poset $P_2$ of figure \ref{pp2}.
For $A=\{ a\}$, we have $\Diamond A = \{ a, b, c \}$, and $\Diamond^2 A = P_2$.

\begin{figure} [H]
\begin{center}

\begin{tikzpicture}

    \tikzstyle{every node}=[draw, circle, fill=white, minimum size=4pt, inner sep=0pt, label distance=1mm]

    \draw (0,0) node (a) [label=left:$a$] {}
        -- ++(270:1cm) node (b) [label=left:$b$] {}
        -- ++(45:1.4142cm) node (c) [label=left:$c$] {}
        -- ++(270:1cm) node (d) [] {}
        -- ++(45:1.4142cm) node (e) [] {};

\end{tikzpicture}

\end{center}
\caption{\label{pp2} A five-element partial order}
\end{figure}
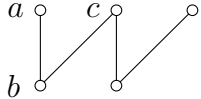




\noindent Continuing in this way, we get models where $\Diamond$, $\Diamond^{2}$, \dots, $\Diamond^{n}$ are different modalities,
for any fixed $n \geq 1$.

Similarly, it is possible to get a model where $\Box$, $\Box^{2}$, \dots, $\Box^{n}$ are different modalities, for any fixed $n \geq 1$.
Take $P_n$ as before and $A$ the complement of the rightmost two-elements in $P_n$.
Note, also, that taking $P_\infty$ it can be seen that all the modalities $\Box^{i}$ are different,
a goal that is also achieved taking the disjoint union of all the (finite) $P_n$. 
Finally, as also $D$ exists in the given models and we have seen in Proposition \ref{DgivesBox} that in that case $\Box$ and $\neg D$ coincide,
we have also shown that the modalities $(\neg D)^i$ are all different. 

In order to compare $\Box$ and $\Diamond$ with $\neg D$ and $D\neg$, respectively,
let us consider the infinite meet-complemented distributive lattice $1 \oplus (\mathbb{N} \times \mathbb{N})^{\partial}$,
already given in Figure \ref{G}.
We have that both $\Box$ and $\Diamond$ exist.
However, $D$ does not exist.

In the following proposition we prove two inequalities stated in \cite[p. 302]{Dunn}.

\begin{prop} \label{Dunn}
 Let us consider a lattice of $\mathbb{ML_\textrm{d}^{\Box \Diamond}}$. Then,

 {\bf{(D1)}} $\Diamond a \wedge \Box b \leq \Diamond (a \wedge b)$ and {\bf{(D2)}} $\Box (a \vee b) \leq \Box a \vee \Diamond b$.
\end{prop}

\begin{proof}
 (D1) By ($\Box$E) we have (i) $b \vee \neg \Box b = 1$.
 We also have that (ii) $\neg(a \wedge b) \wedge b \leq \neg a \leq \neg a \vee \neg \Box b$.
 Also, $\neg (a \wedge b) \wedge \neg \Box b \leq \neg a \vee \neg \Box b$.
 Then, from (i), (ii), and (iii) it follows, using Lemma \ref{dist}, that (iv) $\neg (a \wedge b) \leq \neg a \vee \neg \Box b$.
 Now, by ($\Diamond$I), we have that $\neg (a \wedge b) \vee \Diamond (a \wedge b) = 1$.
 So, using (iv), it follows that $\neg a \vee (\Diamond (a \wedge b) \vee \neg \Box b) = 1$.
 So, by ($\Diamond$E), it follows that $\Diamond a \leq \Diamond (a \wedge b) \vee \neg \Box b$.
 By monotonicity of $\wedge$ it follows that $\Diamond a \wedge \Box b \leq (\Diamond (a \wedge b) \vee \neg \Box b) \wedge \Box b$,
 which implies, by distributivity, that $\Diamond a \wedge \Box b \leq \Diamond (a \wedge b)$.

 (D2) By ($\Diamond$I) we have (i) $\neg b \vee \Diamond b = 1$.
 We also have that (ii) $\Box(a \vee b) \wedge \neg b \leq \Box a \leq \Box a \vee \Diamond b$,
 which follows from $a \vee \neg (\Box (a \vee b) \wedge \neg b) = 1$,
 which follows from $(a \vee b) \vee \neg \Box (a \vee b) = 1$, which holds because of ($\Box$E).
 Now, we also have (iii) $\Box (a \vee b) \wedge \Diamond b \leq \Diamond b \leq \Box a \vee \Diamond b$.
 Finally, using Lemma \ref{dist} with (i), (ii), and (iii), we get (D2).
\end{proof}

Since $\mathbb{ML^{\Box \Diamond}}$ is a variety, its distributive extension $\mathbb{ML_\textrm{d}^{\Box \Diamond}}$
is a subvariety of it, which has a forgetful functor on the variety of bounded distributive lattices.
Since the variety of bounded distributive lattices is generated by its finite elements,
it naturally arises the question if ${\mathbb{ML_\textrm{d}^{\Box \Diamond}}}$ is also generated by its finite members.
We end this section answering this question.

Let {\bf{L}} be a bounded distributive lattice, $X$ a finite non void subset of $L$ and
{\bf{L$_X$}} the (finite) sublattice of {\bf{L}} generated by $X \cup \{0,1\}$.
Since {\bf{L$_X$}} is a finite distributive lattice,
it has its own meet-complement $\neg'x := \textrm{max}\{y \in L_X : x \wedge y = 0 \}$.
By a classical argument (\cite[Thm. 11.9.1]{DunnH}),
we have the following result relating the meet-complements of an element in $L_X$ and in $L$, respectively.

\begin{lem} \label{cero}
Let {\bf{L}} be a bounded distributive lattice, and $X$ and $\mathbf{L}_X$ are as before.
If $x \in L$ has  meet-complement in $L$, $\neg x$, and $x, \neg x \in X$, then $\neg' x = \neg x$.
\end{lem}

\begin{proof}
For any $x \in L_X$, we have that $\{ y \in L_X : x \wedge y = 0 \} \subseteq \{ y \in L : x \wedge y = 0 \}$.
Hence, $\neg' x \leq \neg x$.
On the other hand, take $x \in X$.
Since $\neg x$ is also in $X$, and hence in $LX$, and $\neg x \wedge x = 0$, we conclude that $\neg x \leq \neg' x$.
\end{proof}

Note also that since {\bf{L$_X$}} is finite and distributive, {\bf{L$_X$}}$\in \mathbb{ML_\textrm{d}^{\Box \Diamond}}$.
Write $\Box'$ and $\Diamond'$ for the modal operations defined in $L_X$.
We want to relate them with operations $\Box$ and $\Diamond$ in $L$, when they are defined.
We do this in the next two lemmata.

\begin{lem} \label{uno}
Let {\bf{L}} be a bounded distributive lattice, and $X$ and {\bf{L$_X$}} as before.
If $\Box x$ exists in $L$ and $x$, $\Box x$, and $\neg \Box x$ are in $X$, then $\Box' x = \Box x$.
\end{lem}
\begin{proof}
Since for every $y \in L_X$ we have $\neg' y \leq \neg y$, it follows that
\[
\{y \in L_X : x \vee \neg' y = 1\} \subseteq \{y \in L_X : x \vee \neg y = 1\} \subseteq \{y \in L : x \vee \neg y = 1\}.
\]


\noindent Hence, $\Box' x \leq \Box x$.
On the other hand, since $\neg \Box x \in L_X$, we have that $\Box x \in \{y \in L_X : x \vee \neg' y = 1\}$,
and in consequence, $\Box x \leq \Box' x$.
\end{proof}

\begin{lem} \label{dos}
Let {\bf{L}} be a bounded distributive lattice, and $X$ and {\bf{L$_X$}} as before.
If $\Diamond x$ exists in $L$ and $x$, $\neg x$, and $\Diamond x$ are in $X$, then $\Diamond' x = \Diamond x$.
\end{lem}

\begin{proof}
Since $\neg x \in L_X$, we have, by Lemma \ref{cero}, $\neg' x \leq \neg x$.
Hence,
\[
\{y \in L_X : \neg' x \vee  y = 1\} \subseteq \{y \in L_X : \neg x \vee y = 1\} \subseteq \{y \in L : \neg x \vee  y = 1\}.
\]
Then, $\Diamond x \leq \Diamond' x$.

On the other hand, since $\Diamond x \in L_X$, we have that $\neg' x \vee \Diamond x = \neg x \vee \Diamond x = 1$,
and in consequence, $\Diamond' x \leq \Diamond x$.
\end{proof}

Let $t_1$ and $t_2$ be two terms in the signature $\mathcal{S} := \{\wedge, \vee, \neg, \Box, \Diamond, 0, 1\}$.
We may consider both terms depending of the same finite set of variables, say, $V$.
Write $\mathbf{Term}_\mathcal{S}$ for the algebra of terms in the signature $\mathcal{S}$.
Then, we claim that any equation in $\mathcal{S}$ that has a counterexample in ${\mathbb{ML_\textrm{d}^{\Box \Diamond}}}$,
has a finite counterexample.

\begin{prop}
Let $t_1$ and $t_2$ be two terms in the signature $\mathcal{S}$, {\bf{L}} $\in {\mathbb{ML_\textrm{d}^{\Box \Diamond}}}$,
and $h$ an $\mathcal{S}$-homomorphism from $\mathbf{Term}_\mathcal{S}$ to $\mathbf{L}$ such that $h(t_1) \neq h(t_2)$.
Then there exists a finite element {\bf{L$'$}} $\in {\mathbb{ML_\textrm{d}^{\Box \Diamond}}}$ and an $\mathcal{S}$-homomorphism
$h' : \mathbf{Term}_\mathcal{S} \to \mathbf{L}'$, such that $h'(t_1) \neq h'(t_2)$.
\end{prop}

\begin{proof}
For any term $t \in \mathrm{Term}_\mathcal{S}$, let us write $\mathrm{Sub}(t)$ for the set of subterms of $t$.
Write $S_0 := \mathrm{Sub}(t_1) \cup \mathrm{Sub}(t_2)$ and $X : = h(S_0) \cup \neg h(S_0) \cup \{0, 1\} \subseteq L$.
Every term in $X$ can only depend on a finite set of variables, say, $V$, which are the variables that appear either in
$t_1$ or in $t_2$. Note also that $X$ contains the negation of every element of $h(S_0)$.
Now, take $L' = L_X$ and $h'$ the unique $\mathcal{S}$-homomorphism from $\mathrm{Term}_\mathcal{S}$ to $L'$ such that
$h'(x_i) = h(x_i)$, if $i \in V$, and $h'(x_i) = 1$, if $i \notin V$.
>From Lemmata \ref{cero}, \ref{uno}, and \ref{dos}, it follows straightforwardly that $h'(t_1) \neq h'(t_2)$.
\end{proof}

\begin{cor}
\label{FMP}
The variety $\mathbb{ML_\textrm{d}^{\Box \Diamond}}$ is generated by its finite elements.
\end{cor}

Regarding Boolean elements we have the following fact.

\begin{lem} \label{ldBD}
 Let $\bf{A} \in \mathbb{ML_\textrm{d}^{\Box \Diamond}}$ and $a \in A$.
 Then, the following are equivalent:

 (i) $a$ is Boolean,

 (ii) $\Box^{n} a = a$, for all $n \in \mathbb{N^{\ast}}$,

 (iii) $\Diamond^{n} a = a$, for all $n \in \mathbb{N^{\ast}}$.
\end{lem}

\begin{proof}
 It follows immediately using Lemmas \ref{lBB}, \ref{lDB}, and \ref{g(DN)}.
\end{proof}

\section{The S-extension}

We have the following result even in the non-distributive case.

\begin{prop} \label{Boxeq}
 Let {\bf{A}}$\in \mathbb{ML^{\Box \Diamond}}$.
 Then, for any $a \in A$, the following are equivalent:

 (i) $\Diamond \Box a \leq \Box a$,

 (ii) $\Box a \leq \Box \Box a$,

 (iii) $\Box a \vee \neg \Box a = 1$,

 (iv) $\Box \Box a = \Box a$.
\end{prop}

\begin{proof}
 To see that (i), (ii), and (iii) are equivalent to each other just consider that we have both
 $\circ \leq \Box$ iff $\Diamond \leq \circ$ and $\circ \leq \Box$ iff $\circ \vee \neg = 1$,
 due to (A) on the one hand and ($\Box$I) and ($\Box$E) taken together on the other hand, respectively.
 To see that (iv) is equivalent to (ii),
 just remember that we have the inequality $\Box \Box \leq \Box$ for free (see Lemma \ref{lmcln}(vii)).
\end{proof}

We analogously get the dual, which we state without proof.

\begin{prop} \label{Diaeq}
 Let {\bf{A}}$\in \mathbb{ML^{\Box \Diamond}}$.
 Then, for any $a \in A$, the following are equivalent:

 (i) $\Diamond a \leq \Box \Diamond a$,

 (ii) $\Diamond \Diamond a \leq \Diamond a$,

 (iii) $\Diamond a \vee \neg \Diamond a = 1$,

 (iv) $\Diamond \Diamond a = \Diamond a$.
\end{prop}

\begin{rem}
 Note that (iii) of Proposition \ref{Boxeq} says that the algebraic version of \emph{Tertium non datur} is available for $\Box$-formulas.
 Analogously, in the case of (iii) of Proposition \ref{Diaeq} for $\Diamond$-formulas.
\end{rem}

Now, by selecting the first inequalities of propositions \ref{Boxeq} and \ref{Diaeq},
we see that all the schemas in those propositions are, in fact, equivalent to each other.

\begin{prop} \label{pS}
 Let {\bf{A}}$\in \mathbb{ML^{\Box \Diamond}}$.
 Then, the following two schemas are equivalent:

 (i) $\Diamond \Box \leq \Box$,

 (ii) $\Diamond \leq \Box \Diamond$.
\end{prop}

\begin{proof}
 (i) $\Rightarrow$(ii) Applying the schema in (i) to the argument $\Diamond$ gives
 $\Diamond \Box \Diamond \leq \Box \Diamond$.
 Now, we also have $\Diamond \Box \Diamond = \Diamond$, which is Lemma \ref{BD}(v).

 \noindent (ii) $\Rightarrow$(i) Just dualize the previous proof.
\end{proof}

\begin{rem}
 Cases (i) and (ii) in the previous proposition are the algebraic versions of the S5 axioms in modal logic.
\end{rem}

We define the S-extension by adding to $\mathbb{ML^{\Box \Diamond}}$ the algebraic version of the modal logic S4-schema:

\smallskip

{\bf{(S)}} $\Box \leq \Box \Box$.

\smallskip

\noindent It is clear that all the schemas appearing in propositions 13-14 hold in the S-extension.
In particular, it is clear that it is equivalent to take $\Diamond \Diamond \leq \Diamond$ instead of (S).

We will occasionally consider the following generalization.

\begin{defi}
 For any $n \in \mathbb{N^{\ast}}$, the $S^{n}$-extension is the extension of $\mathbb{ML^{\Box \Diamond}}$
with $\Box^{n+1}= \Box^n$.
\end{defi}

Note that $\Box$ still exists in the pentagon.
So, expanding meet-complemented lattices with $\Box$ having (S) does not imply modularity,
and so, it does not imply distributivity.
Also, we still do not have $\neg \Box \neg \leq \Diamond$,
as may be seen in the pentagon taking the argument to be the non-atom-coatom.

Let us use the notation $\mathbb{ML_\textrm{S}^{\Box \Diamond}}$
for the class of meet-complemented lattices with $\Box$ and $\Diamond$ that satisfy (S).

The following result will be useful.

\begin{lem} \label{lS}
 Let us consider a lattice of ${\mathbb{ML_\textrm{S}^{\Box \Diamond}}}$. Then,
 (i) $\Diamond \neg \Box \leq \neg \Box$,
 (ii) $\Box \neg \Box = \neg \Box$,
 (iii) $\Box \Diamond \neg \Box = \neg \Box$,
 (iv) $\neg \Box \neg = \Box \Diamond$,
 (v) $\neg \Box \Diamond = \Box \neg$.
\end{lem}

\begin{proof}
 (i) Using Proposition \ref{NBD}(iii) we have $\Diamond \neg \Box \leq \neg \Box \Box$.
 Now, using (S) we have $\neg \Box \Box = \neg \Box$.

 \noindent (ii) The inequality $\Box \neg \Box \leq \neg \Box$ holds because of Lemma \ref{lmcln}(ii).
 To get the other inequality we use ($\Box$I), for which it is enough to have $\neg \Box \vee \neg \neg \Box =1$,
 which follows in the S extension as we have $\Box \vee \neg \Box=1$ and $\circ \leq \neg \neg$.

 \noindent (iii) Applying $\Box$-monotonicity to (i) we get $\Box \Diamond \neg \Box \leq \Box \neg \Box$.
 Applying (ii) we get $\Box \Diamond \neg \Box \leq \neg \Box$.
 To get the other inequality apply (B1).

 \noindent (iv) Applying $\neg$-antimonotonicity to (B1), we get $\neg \Box \Diamond \leq \neg$.
 Applying $\Box$-monotonicity we get $\Box \neg \Box \Diamond \leq \Box \neg$.
 Applying $\neg$-antimonotonicity again, we get $\neg \Box \neg \leq \neg \Box \neg \Box \Diamond$,
 which, using(ii), gives $\neg \Box \neg \leq \neg \neg \Box \Diamond$, which, using $\neg \neg \Box = \Box$
 (see Lemma \ref{lmcln}(v)), finally gives $\neg \Box \neg \leq \Box \Diamond$.
 The other inequality appears in Proposition \ref{NBD}(v).

 \noindent (v) Applying $\neg$-antimonotonicity to Proposition \ref{pS}(iii), we get $\neg \Box \Diamond \leq \neg \Box$.
 Then, by Proposition \ref{NBD}(ii), we have $\neg \Box \Diamond \leq \Box \neg$.
 The other inequality appears in Proposition \ref{NBD}(vi).
\end{proof}

It is natural to consider what happens now with the modalities.
With that aim, consider the following results.
The next proposition concerns the positive modalities.

\begin{prop}
  Let us consider a lattice of ${\mathbb{ML_\textrm{S}^{\Box \Diamond}}}$. Then,
 (i) $\Diamond \Box \leq \circ$,
 (ii) $\Diamond \Box \leq \Box$,
 (iii) $\Diamond \Box \leq \Diamond \Box \neg \neg$,
 (iv) $\circ \leq \neg \neg$,
 (v) $\Box \leq \Box \neg \neg$,
 (vi) $\Diamond \Box \neg \neg \leq \Box \neg \neg$,
 (vii) $\Diamond \Box \neg \neg \leq \Diamond$,
 (viii) $\Box \neg \neg \leq \neg \neg$,
 (ix) $\Diamond \leq \Box \Diamond$,
 (x) $\neg \neg \leq \Box \Diamond$.
\end{prop}

\begin{proof}
 \noindent (i) is (B2).

 \noindent (ii) is Proposition \ref{pS}(ii).

 \noindent (iii) follows from $x \leq \neg \neg$ using monotonicity of both $\Box$ and $\Diamond$.

 \noindent (iv) is clear.

 \noindent (v) follows from $\circ \leq \neg \neg$ using $\Box$-monotonicity.

 \noindent (vi) follows by Proposition \ref{pS}(ii).

 \noindent (vii) Using ($\Diamond$E) it is enough to prove $\neg \Box \neg \neg \vee \Diamond = 1$.
 Now, by ($\Diamond$I), we have $\neg \vee \Diamond = 1$ and using Lemma \ref{lmcln}(ii),
 we may get $\neg \leq \neg \Box \neg \neg$.

 \noindent (viii) is a particular case of Lemma \ref{lmcln}(ii).

 \noindent (xi) is Proposition \ref{pS} (iii).

 \noindent (x) is Lemma \ref{BD}(i).
\end{proof}

It can also be seen that the reverse inequalities are not the case.
For that purpose, as before, $a, b, c$ will be the atom non-coatom, the atom-coatom,
and the coatom non-atom of the pentagon, respectively.
Also, $l$ and $m$ will be any of the atoms and the coatom of the (distributive) lattice $2^{2} \oplus 1$, respectively.
We now consider each case of the previous proposition.
(i) Take $l$ or $m$.
(ii) Take $a$.
(iii) Take $m$.
(iv) Take $m$.
(v) Take $m$.
(vi) Take $c$.
(vii) Take $l$.
(viii) Take $l$.
(ix) Take $c$.
(x) Take $l$.

We also need to show the cases where the elements are incomparable.
That is, using the symbol $\parallel$ for non-comparability,
we need to show that $\circ \parallel \Box$, $\circ \parallel \Diamond \Box \neg \neg$, $\circ \parallel \Box \neg \neg$, $\circ \parallel \Diamond$,
$\Box \parallel \Diamond \Box \neg \neg$, $\Box \parallel \Diamond$, $\Box \neg \neg \parallel \Diamond$, and $\neg \neg \parallel \Diamond$.
That means we have to falsify sixteen inequalities.
We now see it is enough to falsify the half.

The node $l$ falsifies both $\circ \leq \Box \neg \neg$ and $\Diamond \leq \neg \neg$,
the first of which implies the falsification of both $\circ \leq \Box$ and $\circ \leq \Diamond \Box \neg \neg$,
and the second the falsification of all three  $\Diamond \leq \circ$, $\Diamond \leq \Box \neg \neg$, and $\Diamond \leq \Box$.

The node $m$ falsifies $\Box \neg \neg \leq \circ$, $\Diamond \Box \neg \neg \leq \circ$ , and $\Diamond \Box \neg \neg \leq \Box$.

The node $a$ falsifies $\Box \leq \circ$.

The node $c$ falsifies $\circ \leq \Diamond$ and $\Box \leq \Diamond$,
the first of which implies the falsification of $\neg \neg \leq \Diamond$,
and the second the falsification of both $\Box \leq \Diamond \Box \neg \neg$ and $\Box \neg \neg \leq \Diamond$.

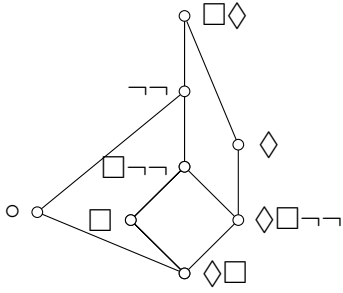
\begin{figure} [ht]
\begin{center}

\begin{tikzpicture}

  \tikzstyle{every node}=[draw, circle, fill=white, minimum size=4pt, inner sep=0pt, label distance=1mm]

    \draw (0,0) node (BD) [label=right:$\Box \Diamond$] {}
        -- ++(270:1cm) node (nn) [label=left:$\neg \neg$] {}
        -- ++(270:1cm) node (Bnn) [label=left:$\Box \neg \neg$] {}
        -- ++(225:1cm) node (B) [label=left:$\Box$] {}
        -- ++(315:1cm) node (DB) [label=right:$\Diamond \Box$] {}
        -- ++(45:1cm) node (DBnn) [label=right:$\Diamond \Box \neg \neg$] {}
        -- ++(90:1cm) node (D) [label=right:$\Diamond$] {}
        -- (BD)
        -- ++(270:2cm) node (Bnn) [] {}
        -- ++(225:1cm) node (B) [] {}
        -- ++(315:1cm) node (DB) [] {}
        -- ++(157.25:2.1cm) node (c) [label=left:$\circ$] {}
        -- (nn);

    \draw (Bnn) -- (DBnn);

\end{tikzpicture}

\end{center}
\caption{\label{pm} The positive modalities in the S-extension}
\end{figure}





The next proposition concerns the negative modalities.

\begin{prop}
 Let us consider a lattice of ${\mathbb{ML_\textrm{S}^{\Box \Diamond}}}$. Then,
 (i) $\Diamond \Box \neg \leq \Box \neg$,
 (ii) $\Diamond \Box \neg \leq \Diamond \neg$,
 (iii) $ \Box \neg \leq \neg$,
 (iv) $\neg \leq \Box \Diamond \neg$,
 (v) $\Diamond \neg \leq \Box \Diamond \neg$,
 (vi) $\Diamond \neg \leq \Diamond \neg \Box$,
 (vii) $\Box \Diamond \neg \leq \neg \Box$,
 (viii) $\Diamond \neg \Box \leq \neg \Box$.
\end{prop}

\begin{proof}
 (i) follows using (A) from $\Box \neg \leq \Box \Box \neg$, which holds using (S).

 \noindent (ii) follows applying $\Diamond$-monotonicity to $\Box \neg \leq \neg$,
 which is Lemma \ref{lmcln}(ii).

 \noindent (iii) is Lemma \ref{lmcln}(ii).

 \noindent (iv) follows from $\Diamond \neg \leq \Diamond \neg$ using (A).

 \noindent (v) follows using (A) from $\Diamond \Diamond \neg \leq \Diamond \neg$, which holds in the S extension.

 \noindent (vi) follows applying $\Diamond$-monotonicity to $\Diamond \neg \leq \neg \Box$,
 which is Proposition \ref{NBD}(iii), and then using $\Diamond \Diamond =\Diamond$.

 \noindent (vii) follows applying $\Box$-monotonicity to $\Diamond \neg \leq \neg \Box$,
 which is Proposition \ref{NBD}(iii), and then use $\Box \neg \Box = \neg \Box$ (see Lemma \ref{lS}(ii)).

 \noindent (viii) As $\Diamond \neg \leq \neg \Box$ (see Proposition \ref{NBD}(iii)),
 we have $\Diamond \neg \Box \leq \neg \Box \Box$.
 The result follows using (S).
\end{proof}

As in the case of the positive modalities, we may see that the reverse inequalities are not the case.
For that purpose, we use again $a, b, c, l, m$ as in the positive case.
We now consider each case of the previous proposition.
(i) Take $b$.
(ii) Take $l$.
(iii) Take $l$.
(iv) Take $l$.
(v) Take $b$.
(vi) Take $m$.
(vii) Take $m$.
(viii) Take $b$.

We also need to see (i) that both $\Box \neg$ and $\neg$ are incomparable with $\Diamond \neg$ and
(ii) that $\Box \Diamond \neg$ and $\Diamond \neg \Box$ are incomparable.
In order to see (i) it is enough, as $\Box \neg \leq \neg$, to see that $l$ falsifies
$\Diamond \neg \leq \Box \neg$ and that $b$ falsifies $\neg \leq \Diamond \neg$.
In order to get (ii), we easily see that $b$ falsifies $\Box \Diamond \neg \leq \Diamond \neg \Box$
and $m$ falsifies the reciprocal inequality.

\begin{figure} [ht]
\begin{center}

\begin{tikzpicture}

  \tikzstyle{every node}=[draw, circle, fill=white, minimum size=4pt, inner sep=0pt, label distance=1mm]

    \draw (0,0)            node       [label=left:$\Box \Diamond \neg$] {}
        -- ++(225:1cm)     node (c)   [label=left:$\neg$] {}
        -- ++(270:1cm)     node (Bn)  [label=left:$\Box \neg$] {}
        -- ++(315:1cm)     node (DBn) [label=right:$\Diamond \Box \neg$] {}
        -- ++(60:1.4142cm) node (Dn)  [label=right:$\Diamond \neg$] {}
        -- ++(30:1cm)      node (DnB) [label=right:$\Diamond \neg \Box$] {}
        -- ++(120:1.35cm)   node (nB)  [label=right:$\neg \Box$] {}
        -- (0,0);

    \draw (0,0) -- (Dn);
\end{tikzpicture}

\end{center}
\caption{\label{nm} The negative modalities in the S-extension}
\end{figure}
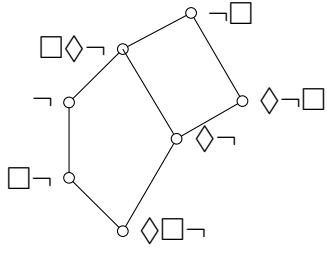

Now that we have seen that we have at least the inequalities indicated, we need to show that there are no more.

\begin{prop}
 The only modalities are the ones appearing in figures \ref{pm} and \ref{nm}.
\end{prop}

\begin{proof}
 All the posible (four) modalities of length one, that is $\circ$, $\neg$, $\Box$, and $\Diamond$ have been considered.

 Regarding the nine possible modalities of length two, it is easily seen that we may eliminate $\neg \Diamond$,
 because it equals $\Box \neg$ and that we may eliminate both $\Box \Box$ and $\Diamond \Diamond$,
 because we are dealing with modalities in the S extension.
 The remaining six modalities of length two have already been considered.

 Regarding the twenty seven possible modalities of length three, it is easy to eliminate twenty three of them
 considering the following well established equalities: $\Box \Box = \Box$, $\Diamond \Diamond = \Diamond$,
 $\neg \neg \neg = \neg$, $\neg \Diamond = \Box \neg$, $\Diamond \Box \Diamond = \Diamond$, $\Box \Diamond \Box = \Box$,
 $\neg \neg \Box = \Box$, $\Diamond \neg \neg = \Diamond$, $\Box \neg \Box = \neg \Box$, $\neg \Box \neg = \Box \Diamond$,
 and $\neg \Box \Diamond = \Box \neg$.
 The four remaining ones, that is, $\Box \neg \neg$, $\Box \Diamond \neg$, $\Diamond \neg \Box$, and $\Diamond \Box \neg$,
 have already been considered.

 Regarding the possible modalities of length four, they either begin or end with one of the four remaining combinations
 of length three already considered. This gives twenty-four length four modalities to be considered.
 Using previous mentioned equalities and also that $\neg \Box \Diamond = \Box \neg$ and $\Box \Diamond \neg \Box = \neg \Box$,
 the only new modality is $\Diamond \Box \neg \neg$, which we have already considered as positive modality.

 Regarding possible modalities of length five, they either begin or end with the only modality of length four.
 That gives six possible modalities.
 However, all of them may be reduced using already mentioned equalities.
\end{proof}

We will use the following in the next section.

\begin{lem} \label{BnB}
 Let {\bf{A}}$\in \mathbb{ML^{\Box \Diamond}}$ in the $S^{n}$-extension.
 Then, for any $a \in A$, $\Box^{n}a$ is Boolean.
\end{lem}

\begin{proof}
 Immediate.
\end{proof}


\section{The distributive S-extension}

Now, let us extend with both distributivity and (S),  using the notation $\mathbb{ML_\textrm{dS}^{\Box \Diamond}}$
for the class of distributive meet-complemented lattices with $\Box$ and $\Diamond$ that satisfy (S).
We will see that, as expected, the number of modalities diminishes.

\begin{prop}
  Let us consider a lattice of ${\mathbb{ML_\textrm{dS}^{\Box \Diamond}}}$. Then,
 (i) $\Diamond \neg \Box=\neg \Box$,
 (ii) $\neg \neg \Diamond \leq \Diamond$,
 (iii) $ \neg \Box \neg \leq \Diamond$,
 (iv) (Definability of $\Diamond$) $\Diamond = \neg \Box \neg$.
\end{prop}

\begin{proof}
 (i) As distributivity implies that $a \leq \Diamond a$, we have that $\neg \Box \leq \Diamond \neg \Box$.
 On the other hand, by part (iii) of Proposition \ref{NBD}, we have that $\Diamond \neg \Box \leq \neg \Box \Box$.
 So, in the S extension it follows that $\Diamond \neg \Box \leq \neg \Box$.
 (ii) In the S extension we have $\Diamond \vee \neg \Diamond = 1$.
 Having distributivity, we may derive the result.
 (iii) follows from (ii) and part (ii) of Proposition \ref{NBD}.
 (iv) follows from part (iii) and part (i) of Proposition \ref{NBD}.
\end{proof}

Regarding the modalities,
we have the following positive $$\Box \leq x \leq \neg \neg \leq \Diamond = \neg \Box \neg$$ with also
$$\Box \leq \neg \Diamond \neg = \Box \neg \neg \leq \neg \neg$$ ($x$ and $\neg \Diamond \neg = \Box \neg \neg$ being incomparable),
and the following four negative ones: $$\Box \neg \leq \neg \leq \Diamond \neg = \neg \Box \neg \neg \leq \neg \Box.$$

\noindent That all these nine modalities are different may be seen just considering either an atom or the coatom
of the five-element lattice $2^{2} \oplus 1$.

\begin{figure} [ht]
\begin{center}

\begin{tikzpicture}

  \tikzstyle{every node}=[draw, circle, fill=white, minimum size=4pt, inner sep=0pt, label distance=1mm]

    \draw (0,0) node (nBn) [label=right:$\neg \Box \neg$] {}
        -- ++(270:1cm) node (nn) [label=right:$\neg \neg$] {}
        -- ++(225:1cm) node (c) [label=left:$\circ$] {}
        -- ++(315:1cm) node (B) [label=right:$\Box$] {}
        -- ++(45:1cm) node (Bnn) [label=right:$\Box \neg \neg$] {}
        -- (nn);

\end{tikzpicture}

\end{center}
\caption{\label{pdm} The positive modalities in the distributive S-extension}
\end{figure}
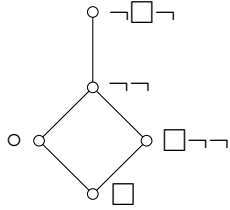







\begin{figure} [ht]
\begin{center}

\begin{tikzpicture}

  \tikzstyle{every node}=[draw, circle, fill=white, minimum size=4pt, inner sep=0pt, label distance=1mm]

    \draw (0,0) node (nB) [label=right:$\neg \Box$] {}
        -- ++(270:1cm) node (nBnn) [label=right:$\neg \Box \neg \neg$] {}
        -- ++(270:1cm) node (n) [label=right:$\neg$] {}
        -- ++(270:1cm) node (Bn) [label=right:$\Box \neg$] {};

\end{tikzpicture}

\end{center}
\caption{\label{ndm} The negative modalities in the distributive S-extension}
\end{figure}
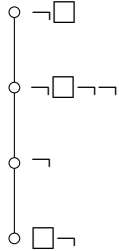






The following equations appear in Proposition 10.2.13 in \cite[p. 361]{DunnH}.

\begin{prop}
  Let us consider a lattice of ${\mathbb{ML_\textrm{dS}^{\Box \Diamond}}}$. Then,
 (i) $\Box (a \wedge \Diamond b)=\Box a \wedge \Diamond b$,
 (ii) $\Diamond (a \vee \Box b) = \Diamond a \vee \Box b$,
 (iii) $\Box ( a \vee \Box b) = \Box a \vee \Box b$, and
 (iv) $\Diamond ( a \wedge \Diamond b) = \Diamond a \wedge \Diamond b$.
\end{prop}

\begin{proof}
 (i) We have, using $\Box$-monotonicity, $\Box (a \wedge \Diamond b) \leq \Box a \wedge \Box \Diamond b$,
 where the rhs, using distributivity, is less or equal to $\Box a \wedge \Diamond b$.
 For the other inequality, let us prove that $\Diamond (\Box a \wedge \Diamond b) \leq a \wedge \Diamond b$ and then use (A).
 By $\Diamond$-monotonicity, we have $\Diamond (\Box a \wedge \Diamond b) \leq \Diamond \Box a, \Diamond \Diamond b$,
 which, by (B1) and (S) give our goal.

 \noindent (ii) Using Lemma \ref{(DN)}(iii) and (S), we have that
 $\Diamond (a \vee \Box b) \leq \Diamond a \vee \Diamond \Box b \leq \Diamond a \vee \Box b$.
 For the other inequality, using $\Diamond$-monotonicity, we have, on the one hand, $\Diamond a \leq \Diamond (a \vee \Box b)$.
 On the other hand, $\Box b \leq a \vee \Box b \leq \Diamond ( a \vee \Box b)$.

 \noindent (iii) Using (D2), we have $\Box(a \vee \Box b) \leq \Box a \vee \Diamond \Box b$.
 And using (S), we get $\Box ( a \vee \Box b) \leq \Box a \vee \Box b$.
 For the other inequality, on the one hand, by $\Box$-monotonicity, we have $\Box a  \leq \Box (a \vee \Box b)$.
 On the other hand, by $\Diamond$-monotonicity, we get $\Box \Box b \leq \Box (a \vee \Box b)$,
 which, by (S), gives $\Box b \leq \Box (a \vee \Box b)$, as desired.

 \noindent (iv) By $\Diamond$-monotonicity, we get both $\Diamond (a \wedge \Diamond b) \leq \Diamond a$ and
 $\Diamond ( a \wedge \Diamond b) \leq \Diamond \Diamond b$,
 where the last inequality, by (S), gives $\Diamond (a \wedge \Diamond b) \leq \Diamond b$.
 For the other inequality, using Proposition \ref{Diaeq}(i), we get $\Diamond a \wedge \Diamond b \leq \Diamond a \wedge \Box \Diamond b$.
 Now, using (D1), we have $\Diamond a \wedge \Box \Diamond b \leq \Diamond ( a \wedge \Diamond b)$.
 The goal follows by $\leq$-transitivity.
\end{proof}

Regarding $B$, the operator defined in Section 2, we can now prove that $B = \Box$.

\begin{prop}
 Let us take a meet-complemented distributive lattice with $\Box$ in the $S^{n}$-extension.
 Then, also $B$ exists, with $B = \Box^{n}$.
\end{prop}

\begin{proof}
 We have $\Box^{n} a \leq a$, which already holds in the distributive context, as seen in Lemma \ref{g(DN)}.
 We also have $\Box^{n} a \vee \neg \Box^{n} a = 1$, as stated in Lemma \ref{BnB}.
 Now, suppose (i) $b \leq a$ and (ii) $b \vee \neg b = 1$.
 From (ii), using Lemma \ref{ldBD}(ii), it follows (iii) $b \leq \Box^{n}b$.
 From (i), applying $n$-times $\Box$-monotonicity, we get (iv) $\Box ^{n}b \leq \Box ^{n}a$.
 From (iii) and (iv), by transitivity, we finally get $b \leq \Box^{n}a$.
\end{proof}

\begin{cor}
 Let us take a meet-complemented distributive lattice with $\Box$ in the $S$-extension.
 Then, also $B$ exists, with $B = \Box$.
\end{cor}

However, we still do not have $\Box = \Delta$, for $\Delta$ as defined in fuzzy logic,
because, taking the coatoms $a$ and $b$ in the lattice $1 \oplus 2^{2}$,
it is not the case that $\Box (a \vee b) = \Box a \vee \Box b$.
\vskip5pt

In what follows, we call \emph{$\Box$-$\Diamond$-algebra} any algebra in $\mathbb{ML_\textrm{d}^{\Box \Diamond}}$
and \emph{S-algebra} any algebra in the S-extension of ${\mathbb{ML_\textrm{d}^{\Box \Diamond}}}$.
It is a well know fact that any finite distributive lattice is the lattice of upsets of a finite poset \cite{DavPr}.
In Section \ref{five}, we have noticed how the lattice of upsets of a finite poset becomes a $\Box$-$\Diamond$-algebra.
In the rest of this section, we characterize those finite posets whose $\Box$-$\Diamond$-algebra of upsets is an S-algebra.
We call this posets, \emph{S-posets}.
\vskip5pt
We start by fixing some terminology.
Let $P$ be a finite poset and $x, y \in P$.
A \emph{zigzag-path} (z-path for short) from $x$ to $y$ in $P$ is a finite sequence of elements in $P$, $(z_0, z_1, ..., z_n)$,
such that $z_0 = x$, $z_n = y$, and for $0 \leq i < n$, either $z_i \leq z_{i+1}$ or $z_{i+1} \leq z_{i}$.
The \emph{length} of a z-path $(z_0, z_1, ..., z_n)$ is the number of elements in the sequence defining the path less one;
$n$ in the given case.
We define the \emph{zigzag distance} between two-elements in $P$ as the minimum of the set of lengths of z-paths between them, if this set is not void,
and $\infty$ if the set is void.
We write $\mathrm{zd}(x,y)$ for the zigzag distance between $x$ and $y$.
For $x$, $y$ in $P$, we say that $x$ and $y$ \emph{are in the same zigzag component} of $P$, whenever $\mathrm{zd}(x,y) \neq \infty$.
The \emph{zigzag component} of $x \in P$ is the set $C(x) := \{y \in P \ : \ \textrm{zd}(x,y) \neq \infty\}$.

\begin{prop}
The $\Box$-$\Diamond$-algebra of a finite poset $P$ is an S-algebra if and only for every pair of elements $x$, $y$ in $P$,
either $\mathrm{zd}(x,y)=\infty$ or $\mathrm{zd}(x,y) \leq 2$.
\end{prop}

\begin{proof}
Without loss of generality we may suppose that $P$ has only one zigzag component,
since the map $A \mapsto \Diamond A = \uparrow \downarrow A$ preserves components.

Assume that  $\mathrm{zd}(x,y) \leq 2$, for every pair of elements $x, y \in P$,
without loss of generality since we cannot have $\mathrm{zd}(x,y)=\infty$.
For $A = \ \uparrow\{a\}$, we have that $\Diamond A = \ \uparrow \downarrow \uparrow\{a\} = P$,
since any element of $P$ is at a zigzag distance of 2 or less from $a$.
Hence $\Diamond^2 \uparrow\{a\} = \Diamond P = P = \Diamond \uparrow\{a\}$.

Take now an arbitrary upset $A$ of $P$.
Since $A$ is finite, there are elements $a_1$, ..., $a_n$ in $P$ such
that $\displaystyle A = \bigcup_{i=1}^{n}\uparrow\{a_i\}$ (for example, take the minimals of $A$).
Hence, we get
\[
\Diamond A = \ \uparrow \downarrow A = \ \uparrow \downarrow(\bigcup_{i=1}^{n}\uparrow\{a_i\}) = \bigcup_{i=1}^{n}\uparrow \downarrow \uparrow\{a_i\} = P.
\]
It follows again that $\Diamond^2 A = P = \Diamond A$. In consequence, the $\Box$-$\Diamond$-algebra of $P$ is an S-algebra.

On the other hand, let us assume that for every upset $A$ of $P$, $\uparrow \downarrow A = \uparrow \downarrow \uparrow \downarrow A$.
Take $a, b \in P$ with $\mathrm{zd}(a,b) \neq \infty$.
Hence, there is a z-path $a = p_0, ..., p_n = b$ joining $a$ with $b$ of minimum length.
Hence, either $\mathrm{zd}(a,b) \leq 4$ or there exists $0 \leq i < n - 4$,
such that $p_{i} > p_{i + 1} <  p_{i + 2} > p_{i + 3} < p_{i + 4}$.
Suppose there exists $0 \leq i < n - 4$, such that $p_{i} > p_{i + 1} <  p_{i + 2} > p_{i + 3} < p_{i + 4}$.
Then, $p_{i}, p_{i + 4} \in \uparrow \downarrow \uparrow \downarrow \uparrow p_{i}$.
Since $X = \uparrow p_{i}$ is an upset of $P$,
we have $\uparrow \downarrow \uparrow \downarrow \uparrow p_{i} = \uparrow \downarrow \uparrow p_{i}$,
and in consequence $p_{i}, p_{i + 4} \in \uparrow \downarrow \uparrow p_{i}$.
Hence there is a z-path of length at most 3 joining $p_{i}$ with $p_{i + 4}$,
contradicting the minimality of the z-path $a = p_0, ..., p_n = b$.
Analogously, if $\mathrm{zd}(a,b) \leq 4$,  we arrive to contradictions if we assume $\mathrm{zd}(a,b) = 3$ or $\mathrm{zd}(a,b) = 4$.
\end{proof}

\begin{cor}
The Let $P$ be a finite poset such that each of its zigzag components has either a first or a last element, then $P$ is an S-poset.
\end{cor}

The reciprocal of the Corollary above does not hold since the poset of Figure \ref{nots} is an example of
an S-poset with just one component and without first or last element.

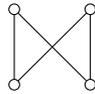
\begin{figure} [H]
\begin{center}

\begin{tikzpicture}
  \tikzstyle{every node}=[draw, circle, fill=white, minimum size=4pt, inner sep=0pt, label distance=1mm]
    \draw (0,0) node (1) [] {}
        -- ++(270:1cm) node (0a) [] {}
        -- ++(45:1.4142cm) node (b) [] {}
        -- ++(270:1cm) node (0b) [] {};
    \draw (1) -- (0b);
\end{tikzpicture}

\end{center}
\caption{\label{nots} A four element partial order}
\end{figure}

\noindent However, we can give the following pictorial characterization of S-posets.

\begin{cor}
Let $P$ be a finite poset, $MP$ the set of maximal elements of $P$ and $mP$ the set of minimal elements of $P$.
Consider $P^0 = MP \cup mP$ as a subposet of $P$.
Then, $P$ is an S-poset if and only if $P^0$ does not have an isomorphic copy of the poset of Figure \ref{N}, as subposet.
\end{cor}

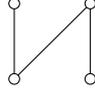
\begin{figure} [H]
\begin{center}

\begin{tikzpicture}

  \tikzstyle{every node}=[draw, circle, fill=white, minimum size=4pt, inner sep=0pt, label distance=1mm]

    \draw (0,0) node (1) [] {}
        -- ++(270:1cm) node (0a) [] {}
        -- ++(45:1.4142cm) node (b) [] {}
        -- ++(270:1cm) node (0b) [] {};

\end{tikzpicture}

\end{center}
\caption{\label{N} A partial order}
\end{figure}



\section{Adding the relative meet-complement}

So far, we have not used the relative meet-complement, which we now define,
in the context of any lattice {\bf{L}} and for any $a, b \in L$ as usual, that is,
$a \to b=max\{c \in L: a \wedge c \leq b \}$.
As very well known and already proved by Skolem (see \cite{Sko1} or \cite{Sko2}),
we automatically get distributivity.
In this section (S) is not required.

We state some results concerning the relationship between $\to$, $\Box$, and $\Diamond$.

\begin{prop} \label{RMC}
 Let us consider a lattice of $\mathbb{ML_\textrm{d}^{\Box \Diamond}}$ with $\to$. Then,

 (1) $\Box(a \to b) \leq \Box a \to \Box b$,

 (2) $\Box(a \to b) \leq \Diamond a \to \Diamond b$,

 (3) $\Diamond a \to \Box b \leq \Box (a \to b)$.
\end{prop}

\begin{proof}
 (1) By a property of $\to$, it is enough to prove that $\Box(a \to b) \wedge \Box a \leq \Box b$.
 Now, also by a property of $\to$, we have that $(a \to b) \wedge a \leq b$.
 So, by monotonicity of $\Box$, it follows that $\Box ((a \to b) \wedge a) \leq \Box b$.
 Then, we reach our goal using that $\Box (a \to b) \wedge \Box a \leq \Box ((a \to b) \wedge a)$,
 which follows from Corollary \ref{bwdv}(i).

 (2) We have that $a \leq \neg(\neg b \wedge (a \to b))$, which, by monotonicity of $\Diamond$, implies that
 $\Diamond a \leq \Diamond \neg(\neg b \wedge (a \to b))$, which, using Lemma \ref{(DN)}(iv), implies that
 $\Diamond a \leq \Diamond \neg \neg y \vee \Diamond \neg (a \to b)$, which, using Lemma \ref{D}, implies that
 $\Diamond a \leq \Diamond b \vee \Diamond \neg (a \to b)$, which, using Proposition \ref{NBD}(iii), implies that
 $\Diamond a \leq \Diamond b \vee \neg \Box (a \to b)$, which finally implies that
 $\Box (a \to b) \leq \Diamond a \to \Diamond b$.

 (3) We have, by ($\Box$E), (i) $(a \to b) \vee \neg \Box (a \to b) = 1$.
 We also have that $\neg \Box(a \to b) \leq \neg \neg \Diamond a$ and $ \neg \Box (a \to b) \leq \neg \Box b$.
 So, (ii) $\neg \Box (a \to b) \leq \neg \neg \Diamond a \wedge \neg \Box b$.
 Then, from (i) and (ii) we get $(a \to b) \vee (\neg \neg \Diamond a \wedge \neg \Box b) = 1$.
 So, $(a \to b) \vee \neg (\Diamond a \to \Box b) = 1$,
 which finally implies, using ($\Box$I), that $\Diamond a \to \Box b \leq \Box (a \to b)$.
\end{proof}

Inequalities (1) and (2) in Proposition \ref{RMC}, (vii) of Proposition \ref{D}, 
one inequality of (ii) of Corollary \ref{bwdv}, and
(3) in Proposition \ref{RMC} are the algebraic versions of the K-axioms IK1-IK5 in \cite[p. 52]{S}, respectively. 

Let us finally note that $\Box$ exists in every subdirectly irreducible Heyting algebra, because
in those algebras, if $a = 1$ then $\Box a = 1$, else $\Box a = 0$.
In fact, the equation $\Box \Box x \approx \Box x$
also holds in every subdirectly irreducible Heyting algebra.
Hence, both in the variety of Heyting algebras with $\Box$ and in the variety of Heyting algebras with $S$-extended $\Box$,
no new equation involving only the Heyting operations holds.










\end{document}